\documentclass{article}

\usepackage{algorithm}
\usepackage[noend]{algorithmic}
\usepackage{amsbsy}
\usepackage{amsfonts}
\usepackage{amsmath}
\usepackage{amssymb}
\usepackage{amsthm}
\usepackage[toc]{appendix}
\usepackage{booktabs}
\usepackage{cite}
\usepackage{fullpage}
\usepackage{graphicx}
\usepackage{hyperref}
\usepackage{makecell,multirow}
\usepackage[square,numbers]{natbib}
\usepackage{tablefootnote}
\usepackage{xspace}

\newcommand{\argmin}{\mathrm{argmin}}
\newcommand{\autobound}{AutoBound\xspace}

\newcommand{\eqdef}{\triangleq}

\newcommand{\ignore}[1]{}
\newcommand{\inner}[2]{ \ensuremath { \left\langle #1, #2 \right\rangle  } }
\newcommand{\integers}{\mathbb{Z}}
\newcommand{\mat}[1]{\mathbf{#1}}
\newcommand{\norm}[1]{\left\|#1\right\|}
\newcommand{\paren}[1]{\left(#1\right)}

\newcommand{\reals}{\mathbb{R}}

\newcommand{\set}[1]{\left\{#1\right\}}

\newcommand{\tee}{\mathsf{\tiny T}}
\renewcommand{\vec}[1]{\mathbf{#1}}

\newtheorem{definition}{Definition}

\newtheorem{lemma}{Lemma}
\newtheorem{proposition}{Proposition}

\newtheorem{theorem}{Theorem}
\newtheorem*{lemma*}{Lemma}
\newtheorem*{theorem*}{Theorem}

\newcommand{\ones}{\mathbf{1}}

\newcommand{\veta}{\boldsymbol{\eta}}
\renewcommand{\a}{\vec{a}}
\renewcommand{\b}{\vec{b}}

\newcommand{\lep}[1]{\underline{#1}}

\newcommand{\rep}[1]{\overline{#1}}

\newcommand{\zeros}{\mathbf{0}}

\newcommand{\safecombination}{\ref{alg:safecombination}\xspace}
\newcommand{\saferate}{\ref{alg:saferate}\xspace}

\newcommand{\mA}{\mat{A}}
\newcommand{\mB}{\mat{B}}

\newcommand{\mU}{\mat{U}}

\newcommand{\mW}{\mat{W}}
\newcommand{\mX}{\mat{X}}

\newcommand{\mZ}{\mat{Z}}
\newcommand{\sA}{\mathcal{A}}

\newcommand{\sI}{\mathcal{I}}

\newcommand{\va}{\vec{a}}
\newcommand{\vb}{\vec{b}}

\newcommand{\ve}{\vec{e}}

\newcommand{\vg}{\vec{g}}

\newcommand{\vp}{\vec{p}}
\newcommand{\vr}{\vec{r}}

\newcommand{\vu}{\vec{u}}
\newcommand{\vv}{\vec{v}}
\newcommand{\vx}{\vec{x}}
\newcommand{\vy}{\vec{y}}
\newcommand{\vz}{\vec{z}}

\newenvironment{varalgorithm}[1]
  {\algorithm\renewcommand{\thealgorithm}{#1}}
  {\endalgorithm}

\newtheorem{innercustomthm}{Theorem}
\newenvironment{customthm}[1]
  {\renewcommand\theinnercustomthm{#1}\innercustomthm}
  {\endinnercustomthm}

\makeatletter
\newcommand{\algrule}[1][.2pt]{\par\vskip.2\baselineskip\hrule height #1\par\vskip.5\baselineskip}
\makeatother

\title{Universal Majorization-Minimization Algorithms\footnote{An earlier version of this work was published as chapter 5 of \url{https://arxiv.org/pdf/2212.11429v1.pdf}.}}
\author{Matthew Streeter\\
\vspace{-.2cm} \\
\small{{\tt \href{mailto:mstreeter@google.com}{mstreeter@google.com}}}}
\date{}

\begin{document}

\maketitle

\begin{abstract}
Majorization-minimization (MM) is a family of optimization methods that iteratively reduce a loss by minimizing a locally-tight upper bound, called a majorizer.  Traditionally, majorizers were derived by hand, and MM was only applicable to a small number of well-studied problems.
We present optimizers that instead derive majorizers \emph{automatically}, using a recent generalization of Taylor mode automatic differentiation.
These \emph{universal} MM optimizers can be applied to arbitrary problems and converge from any starting point, with no hyperparameter tuning.
\end{abstract}

\section{Introduction}

Optimization plays a central role in machine learning and statistics.  To fit a model, one typically minimizes a loss function $f: \reals^n \to \reals$, where $f$ is available in symbolic form, allowing its derivatives to be computed easily via automatic differentiation.

The optimizers used to fit models to data typically fall into one of two categories:
\begin{enumerate}
  \item \emph{General-purpose optimizers based on Taylor polynomials.}  This category includes gradient descent, Newton's method, and many other optimizers.
  \item \emph{Majorization-minimization (MM) optimizers.}  These methods are based on \emph {majorizers} (locally-tight upper bounds), which have been derived by hand for various loss functions of interest.
\end{enumerate}
MM optimizers are attractive because they converge from any starting point, do not require hyperparameter tuning, and sometimes offer faster rates of convergence than general-purpose methods.  
MM optimizers have been derived for logistic regression \cite{bohning1988monotonicity}, quantile regression \cite{hunter2000quantile}, support vector machines \cite{groenen2008svm}, and many other problems \cite{de2016block,lange2016mm}.
However, because each new problem requires a non-trivial derivation, MM has not been applicable to more complex loss functions, such as the ones used in deep learning.

In this paper, we will derive the upper bounds used in MM optimization algorithms \emph{automatically}, thereby creating general-purpose MM optimizers that are as widely applicable as gradient-based methods.
We thereby greatly extend the reach of this decades-old optimization template, making its benefits much more broadly available.
Furthermore, we will see that the resulting optimizers perform well on a variety of problems.

To automatically derive majorizers, we use a recently-developed algorithm \cite{streeter2023automatically} that automatically bounds the Taylor remainder series, using an interval arithmetic variant of Taylor mode automatic differentiation.  Computing a majorizer this way requires an additional forward pass that uses memory linear in the number of inputs.  For this reason, we do not majorize the original loss function (which would require too much memory), but instead seek a majorizer that is valid over a lower-dimensional subspace.

Our first universal MM algorithm, called \saferate, uses a one-dimensional upper bound to derive a learning rate that is guaranteed to monotonically reduce the loss.  This learning rate depends on the current iterate $\vx_t$, and therefore adapts during the course of optimization.  Our second algorithm, \safecombination, generalizes this by computing a linear combination of $d$ update directions rather than a single learning rate.

A possible concern with any such approach is that our automatically-derived upper bounds may be too loose, especially for complex losses.
Reassuringly, we find that the majorizers used by \saferate are tight in practice, even for multi-layer perceptrons with up to 64 hidden layers.

Although our universal MM optimizers are significant in their own right, we hope their biggest impact will be as illustrations of an underappreciated fact:
\begin{quote}
\emph{A general-purpose optimizer can automatically compute polynomial upper and lower bounds on the loss, rather than simply approximate it with a Taylor polynomial.}
\end{quote}

\section{Background}

We begin with a review of MM optimization algorithms, followed by a discussion of recently-developed techniques for bounding the Taylor remainder series.

\subsection{MM Optimization Algorithms}

Majorization-minimization (MM) is a well-known optimization technique based on a simple idea: by minimizing a locally-tight upper bound on a loss function of interest, we can iteratively reduce the loss.  To present the method formally, we will use the following definition.

\begin{definition}[Majorization]
Consider a function $f: \reals^n \to \reals$.  A function $\bar f: \reals^n \times \reals^n \to \reals$ \emph{majorizes} $f$ if, for any $\vy \in \reals^n$:
\begin{equation}
  f(\vx) \le \bar f(\vx, \vy) \quad \forall \vx \in \reals^n
\end{equation}
and furthermore, $\bar f(\vy, \vy) = f(\vy)$.
\end{definition}

If $\bar f$ majorizes $f$, then $\bar f$ provides an upper bound on $f$ that is tight at a specified point $\vy$, and valid everywhere.  Figure~\ref{fig:a_majorizer} illustrates a majorizer for the $\mathrm{softplus}$ function.

\begin{figure}[h]
\begin{center}
\includegraphics[width=0.4\linewidth]{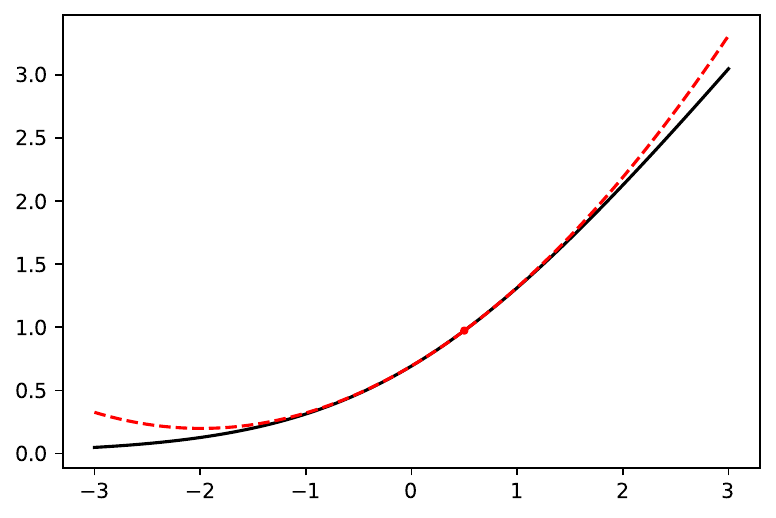}
\caption{The function $f(x) = \log(1+\exp(x))$, and the quadratic majorizer $\bar f(x, y) =f(y) + \nabla f(y)(x-y) + \frac 1 8 (x - y)^2$, evaluated at $y = \frac 1 2$.}
\label{fig:a_majorizer}
\end{center}
\end{figure}

Given an arbitrary starting point $\vx_0 \in \reals^n$, a majorization-minimization algorithm produces a sequence of iterates $\vx_0, \vx_1, \vx_2, \ldots$ by setting
\begin{equation} \label{eq:mm_argmin}
  \vx_{t+1} = \argmin_{\vx \in \reals^n} \set { \bar f(\vx, \vx_t) }.
\end{equation}
An appealing property of MM is that the loss decreases monotonically: $f(\vx_{t+1}) \le f(\vx_t)$ for all $t$.  This follows from the sequence of inequalities:
\begin{equation}
  f(\vx_{t+1}) \le \bar f(\vx_{t+1}, \vx_t) \le \bar f(\vx_t, \vx_t) = f(\vx_t)
\end{equation}
where the first inequality holds because $\bar f$ majorizes $f$, and the second inequality holds by definition of $\vx_{t+1}$.

Majorization-minimization is the subject of a large literature, and many different majorizers have been derived for functions with specific properties.  For example, if one can show that $f$ is $\beta$-smooth (meaning that $\nabla f$ has a Lipschitz constant at most $\beta$), then it can be shown that $f$ is majorized by the quadratic function $\bar f(\vx, \vy) = f(\vy) + \nabla f(\vy) (\vx - \vy) + \frac \beta 2 \norm{\vx - \vy}_2^2$.  However, we are not aware any previous work on automatically deriving majorizers for arbitrary losses.  See \cite{de2016block,lange2016mm} for textbook treatments of MM.

\subsection{Bounding the Taylor Remainder Series} \label{sec:bounding_remainder}

We now briefly review two methods for bounding the Taylor remainder series.
For simplicity, consider a scalar function $f: \reals \to \reals$ (all results in this section generalize to vector-variate functions).  Given a polynomial degree $k$, a \emph{trust region} $[a, b]$, and a scalar $x_0 \in [a, b]$, these methods compute an interval $I$ such that for all $x \in [a, b]$,
\begin{equation} \label{eq:desired_bound}
  f(x) \in \underbrace{ \paren {\sum_{i=0}^{k-1}  \frac {1} {i!} f^{(i)}(x_0) (x - x_0)^i } }_{\text{Degree $k-1$ Taylor polynomial}} + \underbrace{I (x - x_0)^k}_{\text{Remainder bound}}.
  \footnote{The product of an interval $I = [\lep{I}, \rep{I}]$, and a scalar $z$ is defined as $I z \eqdef \set{\alpha z: \alpha \in I} = [ \min \set { \lep{I} z, \rep{I} z} , \max \set { \lep{I} z, \rep{I} z} ]$.}
\end{equation}

A classical method for deriving the interval $I$ is based on the fact that \eqref{eq:desired_bound} holds so long as
\begin{equation} \label{eq:lagrange_remainder_interval}
  I \supseteq \left [\inf_{y \in [a, b]} \set { \frac { f^{(k)}(y) } { (k+1)! } }, \sup_{y \in [a, b]} \set { \frac { f^{(k)}(y) } { (k+1)! } } \right ].
\end{equation}
That \eqref{eq:lagrange_remainder_interval} implies \eqref{eq:desired_bound} can be shown using the Lagrange form of the Taylor remainder series and the mean value theorem \cite{apostol1991calculus}.
Furthermore, an interval that satisfies \eqref{eq:lagrange_remainder_interval} can be computed by first deriving an expression for $f^{(k)}(y)$ (e.g., using automatic differentiation), and then evaluating this expression using interval arithmetic \cite{jaulin2001interval,hansen1979global,hansen2003global,moore1966interval}.

Unfortunately, the bounds produced by direct interval arithmetic evaluation of the $k$th derivative can be very loose.
To address this, we created the AutoBound algorithm \cite{streeter2023automatically}, which uses an interval arithmetic variant of Taylor-mode automatic differentiation to compute bounds that can be much tighter.

\begin{figure}[h]
\begin{center}
\includegraphics[width=0.5\linewidth]{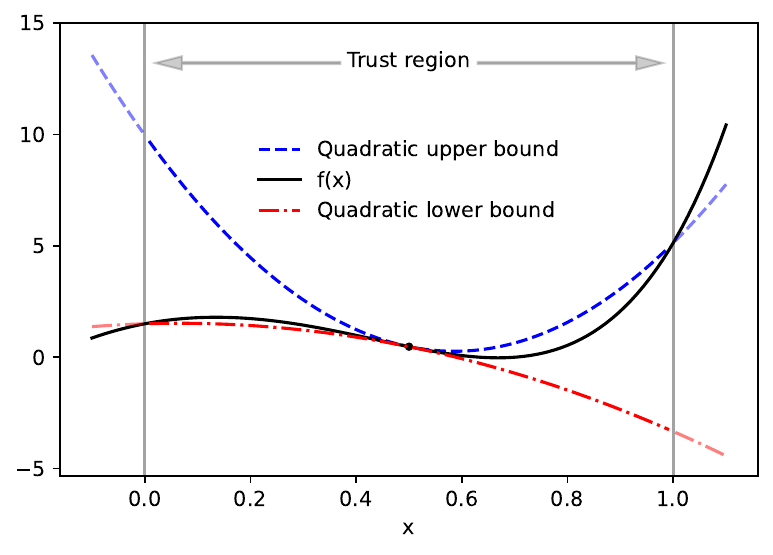}
\caption{Quadratic upper and lower bounds derived by AutoBound \cite{streeter2023automatically} for the function
 $f(x) = \frac 3 2 \exp(3 x) - 25 x^2$, centered at $x_0 = \frac 1 2$, and valid
 over the interval $[0, 1]$.}
\label{fig:example_enclosure}
\end{center}
\end{figure}

Figure~\ref{fig:example_enclosure} gives an example of upper and lower bounds computed by AutoBound when $k = 2$.

\section{Universal MM Optimization Algorithms}

We now present two \emph{universal} MM optimization algorithms: MM optimization algorithms that can be applied to any loss that can be written as a composition of known elementary univariate functions (e.g., $\exp$, $\log$, $\mathrm{relu}$, $\mathrm{softplus}$), plus the binary operations of addition and multiplication.  The term ``universal'' is ours, and to our knowledge these algorithms are the first MM algorithms that are universal in this sense.

In fact, our first algorithm is a special case of our second one, but we choose to present it separately for pedagogical purposes.  Specifically:
\begin{itemize}
  \item The \saferate algorithm uses \autobound to compute a learning rate $\eta_t$ that is guaranteed to reduce the loss on step $t$.
  \item The \safecombination algorithm uses \autobound to compute a vector $\veta_t$ that specifies a combination of user-supplied update directions that is guaranteed to reduce the loss.
\end{itemize}

In both cases, the majorizer is defined over a low-dimensional subspace of the original optimization domain, making the algorithms \emph{majorization-minimization subspace methods}.  In addition, the value of the majorizer will be infinite outside a trust region, whose size affects the tightness of the majorizer.  In order to speed up convergence, both algorithms use a simple heuristic to adapt the size of the trust region over time.

\subsection{Automatically Deriving Majorizers} \label{sec:deriving_majorizers}

In \S\ref{sec:bounding_remainder}, we reviewed two techniques for bounding the Taylor remainder series.  When used to compute a degree $k$ bound for a univariate function $f: \reals \to \reals$, these techniques take as input a trust region $[a, b] \subseteq \reals$ and a point $x_0 \in [a, b]$, and return an interval $I$
such that
\begin{equation} \label{eq:autobound_bound}
  f(x) \in T_{k-1}(x; x_0) + I (x - x_0)^k \quad \forall x \in [a, b]
\end{equation}
where $T_{k-1}(x; x_0) \eqdef \sum_{i=0}^{k-1}  \frac {f^{(i)}(x_0)} {i!} (x - x_0)^i$ is the degree $k-1$ Taylor polynomial of $f$ at $x_0$.

In what follows, we will always have $(x - x_0)^k \ge 0$ for $x \in [a, b]$.  Under this assumption, \eqref{eq:autobound_bound} implies the upper bound
\begin{equation}
  f(x) \le T_{k-1}(x; x_0) + \rep{I} (x - x_0)^k \quad \forall x \in [a, b]
\end{equation}
where $\rep{I}$ is the right end point of $I$.

We can therefore define a majorizer
\begin{equation}
\bar f(x, x_0) \eqdef 
\begin{cases}
  T_{k-1}(x; x_0) + \rep{I} (x - x_0)^k & x \in [a, b] \\
  \infty & x \notin [a, b].
\end{cases}
\end{equation}

For a vector-variate function $f: \reals^n \to \reals$, the situation is similar, except that the scalar product $\rep{I} (x - x_0)^k$ is replaced by an inner product of two tensors with $n^k$ elements (see Appendix~\ref{sec:implementation_details} for details).

\subsection{A Direct Approach} \label{sec:direct_approach}

Given the results from the previous section, the most obvious way to define a universal MM optimizer would be to auto-derive a majorizer for the loss function $f: \reals^n \to \reals$.  This approach can be made to work, and is practical when $n$ is small.

The downside of this approach is that the memory required by AutoBound \cite{streeter2023automatically} scales linearly with $n$, making the direct approach impractical for large-scale optimization problems.  We therefore focus on majorizing $f$ over a lower-dimensional subspace.

\subsection{SafeRate: Determining a Safe Learning Rate} \label{sec:saferate}

\newcommand{\etamax}{\bar \eta}

Iterative methods for optimization, such as gradient descent and Newton's method, optimize a differentiable loss $f: \reals^n \to \reals$ by performing a sequence of updates of the form
\begin{equation}
  \vx_{t+1} = \vx_t + \eta_t \vv_t
\end{equation}
where $\vv_t$ is an update direction and $\eta_t$ is a step size.  To optimize $f$ effectively, we would like to choose $\vv_t$ and $\eta_t$ so as to guarantee convergence to a first-order critical point (or preferably, a local minimum).

Optimization theory provides many methods for choosing $\vv_t$ and $\eta_t$, which have different theoretical guarantees depending on what assumptions are made about $f$.
If $f$ is convex and the gradient norm is bounded, then setting $\vv_t = -\nabla f(\vx_t)$ and $\eta_t = \frac{1}{\sqrt t}$ guarantees that the loss of the average iterate converges to the global minimum \cite{zinkevich2003online}.
If $f$ is non-convex but its gradient has a Lipschitz constant of at most $\beta$ (meaning $f$ is ``$\beta$-smooth''), 
then setting $\vv_t = - \nabla f(\vx_t)$ and $\eta_t = \frac 1 \beta$ guarantees that the loss goes down monotonically \cite{ghadimi2013stochastic}.  See the textbooks \cite{nesterov2003introductory,nocedal1999numerical} for an introduction to optimization theory, including classical results on the rates of convergence achieved by different algorithms under various assumptions.

The loss functions used to train neural networks are neither convex nor $\beta$-smooth.  For such functions, the only globally convergent optimization methods we are aware of involve are classical methods that require line search or trust region adaptation (e.g., \cite[chapters 3-4]{nocedal1999numerical}).  Because these classical approaches are based on Taylor polynomials, and have no way of knowing how accurate the Taylor polynomial approximation is at a given $\vx_t$, they require trial and error in order to find a point that reduces the loss.

Using our new machinery, it is possible to find a learning rate that is \emph{guaranteed} to reduce the loss, at the cost of just one additional forward pass.  To do so, we use \autobound to derive a majorizer of the function
\begin{equation} \label{eq:next_loss}
  h_t(\eta) = f(\vx_t + \eta \vv_t).
\end{equation}

Specifically, given a specified maximum learning rate $\etamax_t$, we auto-derive a majorizer for $h_t$ as described in \S\ref{sec:deriving_majorizers} to obtain a polynomial $P_t$ such that
\begin{equation} \label{eq:next_loss_bound}
   h_t(\eta) \le P_t(\eta) \quad \forall \eta \in [0, \etamax_t].
\end{equation}
Note that $P_t$ defines a majorizer of $f$ that is valid for $\vx \in \set{\vx_t + \eta \vv_t: \eta \in [0, \etamax_t]}$.

We then choose $\eta_t$ to minimize the majorizer, setting
\begin{equation} \label{eq:eta}
  \eta_t = \argmin_{\eta \in [0, \etamax_t]} \set { P_t(\eta) }.
\end{equation}

The $\argmin$ can be found efficiently by computing the roots of a degree $k-1$ polynomial (see Appendix~\ref{sec:implementation_details}).

How do we choose the maximum learning rate, $\etamax_t$?  Because $\eta_t \le \etamax_t$, choosing $\etamax_t$ too small will slow down progress.  On the other hand, as $\etamax_t$ increases, the majorizer $P_t$ becomes looser, so choosing $\etamax_t$ very large will \emph{also} result in a small value of $\eta_t$.

A simple rule, which we will find to be effective in our experiments in \S\ref{sec:mm_experiments}, is to set $\etamax_1$ to an arbitrary value (say, 1), and then to set
\begin{equation} \label{eq:adaptive_eta}
\etamax_{t+1} = \begin{cases}
2 \etamax_t & \mbox { if } \eta_t \ge \frac 1 2 \etamax_t \\
\frac 1 2 \etamax_t & \mbox { otherwise.}
\end{cases}
\end{equation}

We refer to the resulting algorithm as \saferate, and give pseudocode below.

\newcommand{\doracle}{\mathcal{O}}
\newcommand{\states}{\mathcal{S}}

\begin{varalgorithm}{SafeRate}
  \begin{algorithmic}
  \caption{} \label{alg:saferate}
  \STATE {\bf Parameters}:
  \begin{enumerate}
  	\item a function $\doracle$ that, given the current iterate $\vx_t$ and observed gradients $\vg_1, \vg_2, \ldots, \vg_t$, returns an update direction.
	\item a function $\mathrm{PolynomialUpperBound}$ that, given a function $h: \reals \to \reals$, an interval $[a, b]$, and a point $z_0 \in [a, b]$, returns a degree $k$ polynomial $P$ such that $h(z) \le P(z)\ \forall z \in [a, b]$, and $P(z_0) = h(z_0)$. 
  \end{enumerate}
  \algrule
  \STATE {\bf Input}: a loss $f: \reals^n \to \reals$, and an initial point $\vx_1 \in \reals^n$.
  \STATE {\bf Output}: a point $\vx_{T+1} \in \reals^n$.
  \algrule
   \STATE Initialize $\etamax_1 \leftarrow 1$.
   \FOR {$t$ from $1$ to $T$}
    \STATE Set $\vg_t \leftarrow \nabla f(\vx_t)$.
    \STATE Set $\vv_t \leftarrow \doracle(\vx_t, \set{\vg_s}_{s=1}^t)$.
    \STATE Define $h_t(\eta) \eqdef f(\vx_t + \eta \vv_t)$.
    \STATE Set $P_t \leftarrow \mathrm{PolynomialUpperBound}(h_t, [0, \etamax_t], 0)$.
    \STATE Set $\eta_t \leftarrow \argmin_{\eta \in [0, \etamax_t]} \set { P_t(\eta) }$.
    \STATE Set $\vx_{t+1} \leftarrow \vx_t + \eta_t \vv_t$.
    \STATE Set $\etamax_{t+1} \leftarrow \begin{cases}
2 \etamax_t & \mbox { if } \eta_t \ge \frac 1 2 \etamax_t \\
\frac 1 2 \etamax_t & \mbox { otherwise.}
\end{cases}$
  \ENDFOR
  \STATE Return $\vx_{T+1}$.
\end{algorithmic}
\end{varalgorithm}

A few remarks about \saferate are in order:
\begin{enumerate}
  \item As written, \saferate calls \autobound once per iteration, and passes it a symbolic expression that depends on $\vx_t$ and $\vv_t$.  However, in an actual implementation, \autobound is called once up front to return an object (e.g., a TensorFlow tensor) that works for arbitrary $\vx_t$ and $\vv_t$.  Thus, the overhead of recursively analyzing the symbolic expression for the loss $f$ is incurred once up front, rather than $T$ times.
  \item The use of a finite trust region is not necessary for all problems.  In particular, for linear or logistic regression, \autobound is able to compute a quadratic majorizer that is valid over an unbounded trust region.  In such cases we may set $\etamax_1 = \infty$ to effectively disable the use of a finite trust region.
\end{enumerate}

\subsection{SafeCombination: Combining Multiple Update Directions} \label{sec:safe_combination}

\newcommand{\vetamax}{\bar \veta}

The method just discussed for determining a safe learning rate is a special case of a more general method, where a number of possible update directions are provided, and \autobound is used to determine a linear combination of the update directions that is guaranteed to monotonically reduce the loss.

Concretely, let $\mU_t \in \reals^{n \times d}$ be a matrix whose columns specify $d$ possible update directions (e.g., the negative gradient, or the update direction used by Adam or AdaGrad).  We will perform an update of the form:
\begin{equation}
  \vx_{t+1} = \vx_t + \mU_t \veta_t
\end{equation}
for some $\veta_t \in [\vec{0}, \vetamax_t] \subseteq \reals^d$.

The vector $\veta_t$ can be obtained by generalizing equations \eqref{eq:next_loss} and \eqref{eq:next_loss_bound} in the natural way, leading to the bound:
\begin{equation} \label{eq:safecombination_bound}
  f(\vx_t + \mU_t \veta) \le f(\vx_t) + \nabla f(\vx_t)^\tee \mU_t \veta + \veta^\tee \rep{\sI_t} \veta \quad \forall \veta \in [\vec{0}, \vetamax_t]
\end{equation}
for some matrix $\rep{\sI_t} \in \reals^{d \times d}$.

\begin{varalgorithm}{SafeCombination}
  \begin{algorithmic}
  \caption{} \label{alg:safecombination}
  \STATE {\bf Parameters}:
  \begin{enumerate}
  	\item a function $\doracle$ that, given the current iterate $\vx_t$ and observed gradients $\vg_1, \vg_2, \ldots, \vg_t$, returns an $n$ by $d$ update direction matrix.
	\item a function $\mathrm{PolynomialUpperBound}$ that, given a function $h: \reals^d \to \reals$, a vector interval $[\va, \vb] \subseteq \reals^d$, and a point $\vz_0 \in [\va, \vb]$, returns a quadratic polynomial $P$ such that $h(\vz) \le P(\vz)\ \forall \vz \in [\va, \vb]$, and $P(\vz_0) = h(\vz_0)$. 
  \end{enumerate}
  \algrule
  \STATE {\bf Input}: a loss $f: \reals^n \to \reals$, and an initial point $\vx_1 \in \reals^n$.
  \STATE {\bf Output}: a point $\vx_{T+1} \in \reals^n$.
  \algrule
  \STATE Initialize $\vetamax_1 \leftarrow \ones_d \in \reals^d$.
  \FOR {$t$ from $1$ to $T$}
    \STATE Set $\vg_t \leftarrow \nabla f(\vx_t)$.
    \STATE Set $\mU_t \leftarrow \doracle(\vx_t, \set{\vg_s}_{s=1}^t )$.
    \STATE Define $h_t(\veta) \eqdef f(\vx_t + \mU_t \veta)$.
    \STATE Set $P_t \leftarrow \mathrm{PolynomialUpperBound}( h_t, [ \zeros, \vetamax_t ], \zeros)$.
    \STATE Set $\veta_t \leftarrow \mathrm{MinimizeQuadratic}(P_t, [\vec{0}, \vetamax_{t}])$ (see \S\ref{sec:minimizing_quadratic}).
    \STATE Set $\vx_{t+1} \leftarrow \vx_t+ \mU_t \veta_t$.
    \STATE For $i \in \set{1, 2, \ldots, d}$, set $(\vetamax_{t+1})_i \leftarrow \begin{cases}
2 (\vetamax_{t})_i & \mbox { if } (\veta_{t})_i \ge \frac 1 2 (\vetamax_{t})_i \\
\frac 1 2 (\vetamax_{t})_i & \mbox { otherwise.}
\end{cases}$
  \ENDFOR
  \STATE Return $\vx_{T+1}$.
\end{algorithmic}
\end{varalgorithm}

The usual way to minimize the right hand side of \eqref{eq:safecombination_bound} would be to set the derivative to zero, but this does not work here for two reasons.  First, the bound is only valid for $\veta \in [\vec{0}, \vetamax_t]$, and setting the derivative to zero might yield a value outside this hyperrectangle. Second, the matrix $\rep{\sI_t}$ is not necessarily positive semidefinite, so setting the derivative to zero does not necessarily give us the global minimum.  Nevertheless, the right hand side can be approximately minimized over the hyperrectangle $[\zeros, \vetamax_t]$ using a variant of conjugate gradient descent, which we describe in Appendix B.

The vector $\vetamax_t$ is determined adaptively using a scheme similar to \eqref {eq:adaptive_eta}, but applied to each of the $d$ components of $\vetamax_t$ independently. %

It can be shown that, by setting $d = n$ and letting $\mU_t$ be the identity matrix, we obtain a second-order optimizer that can minimize quadratic losses (e.g., least squares linear regression) in one step.

\subsection{Theoretical Guarantees} \label{sec:theory}

Our universal MM algorithms enjoy a number of desirable theoretical guarantees, which we now state.  The proofs are given in Appendix~\ref{sec:proofs}.

First, like other MM optimizers, our universal MM optimizers are guaranteed to monotonically reduce the loss.

\begin{proposition} \label{prop:monotone}
The iterates of \saferate and \safecombination satisfy
\[
  f(\vx_{T+1}) \le f(\vx_T) \le \ldots \le f(\vx_1).
\]
\end{proposition}

The next question we would like to answer is how fast our universal MM optimizers converge.  Because we have not yet developed a complete theory around the tightness of the bounds returned by \autobound, we cannot yet answer this question in full generality.  However, under the standard assumption of a $\beta$-smooth loss, and assuming that the $\mathrm{PolynomialUpperBound}$ function returns the upper bound that follows from the definition of $\beta$-smoothness, we can show our optimizers enjoy convergence guarantees similar to those of backtracking line search \cite{armijo1966minimization}, but without the need to backtrack or to tune line search hyperparameters.  
As is standard in non-convex optimization, the guarantees are stated in terms of the minimum gradient norm.

\newcommand{\thmglobalconvergence}{
Let $f$, $\vv_t$, and $P_t$ be defined as in the pseudocode for \saferate.  Suppose, for some $\beta \ge 1$, $f$ is $\beta$-smooth over the level set $\set{\vy: f(\vy) \le f(\vx_1)}$, and suppose that the $\mathrm{PolynomialUpperBound}$ function returns
\[
  P_t(\eta) = f(\vx_t) + \eta \nabla f(\vx_t)^\tee \vv_t + \frac \beta 2 \eta^2 \paren { f(\vx_t)^\tee \vv_t }^2.
\]
Further suppose $\vv_t = - \nabla f(\vx_t)$.  Then,
\[
  \min_{1 \le t \le T} \set { \norm{ \nabla f(\vx_t) }^2 } \le \frac { 2 \beta \paren { f(\vx_1) - f(\vx_{T+1}) }  } {T}.
\]
The same guarantee holds for \safecombination, provided the negative gradient is one of the $d$ update directions.
}
\newcommand{\thmglobalconvergenceproof}{
By our $\beta$-smoothness assumption, for any $\vz$ with $f(\vz) \le f(\vx_1)$, and for any $t$, we have the quadratic upper bound
\begin{equation}
  f(\vz) \le f(\vx_t) + \nabla f(\vx_t)^\tee (\vz - \vx_t) + \frac \beta 2 \norm{\vz - \vx_t}^2. 
\end{equation}
By Proposition~\ref{prop:monotone}, $f(\vx_{t+1}) \le f(\vx_t)$.  Therefore, taking $\vz = \vx_{t+1}$, and letting $\vg_t = \nabla f(\vx_t)$ for compactness, we have
\begin{align}
  f(\vx_{t+1})
  & \le f(\vx_t) + \vg_t^\tee (-\eta_t \vg_t) + \frac \beta 2 \norm{-\eta_t \vg_t }^2 \nonumber \\
  & = f(\vx_t) - \norm{\vg_t}^2 \paren { \eta_t - \frac {\beta} {2} \eta_t^2 }.
\end{align}
We will show below that for all $t$,
\begin{equation} \label{eq:etat}
  \eta_t = \frac 1 \beta .
\end{equation}
We therefore have
\begin{equation}
  f(\vx_{t+1}) \le f(\vx_t) - \frac {1} {2 \beta} \norm{\vg_t}^2.
\end{equation}
Summing this inequality over all $t$ then proves
\begin{equation} \label{eq:summed_inequality}
  f(\vx_{T+1}) - f(\vx_1) \le - \frac {1} {2 \beta} \sum_{t=1}^T \norm{\vg_t}^2.
\end{equation}
Rearranging \eqref{eq:summed_inequality}, and using $\min_{1 \le t \le T} \set {\norm{\vg_t}^2} \le \frac 1 T \sum_{t=1}^T \norm{\vg_t}^2$ completes the proof.

We now prove \eqref{eq:etat} by induction.  First note that the unconstrained minimizer of $P_t$ is
\begin{equation} \label {eq:unconstrained_eta}
  \argmin_{\eta \in \reals} \set { P_t(\eta) } = \frac 1 \beta.
\end{equation}
Therefore, $\eta_t = \frac 1 \beta$ so long as $\etamax_t \ge \frac 1 \beta$.  For $t = 1$, we have $\etamax_t = 1$, and because $\beta \ge 1$ by assumption, $\frac 1 \beta \le \etamax_1$ and therefore $\eta_1 = \frac 1 \beta$.

Now assume that for some arbitrary $t$, we have $\eta_t = \frac 1 \beta$.  Because $\eta_t \in [0, \etamax_t]$, this implies $\etamax_t \ge \frac 1 \beta$.  Considering the two cases $\eta_t \ge \frac 1 2 \etamax_t$ and $\eta_t < \frac 1 2 \etamax_t$, we see that in either case,
\begin{equation}
  \etamax_{t+1} \ge \frac 1 \beta.
\end{equation}
Therefore, by \eqref{eq:unconstrained_eta}, $\eta_{t+1} = \frac 1 \beta$.
}
\begin{theorem} \label{thm:global_convergence}
\thmglobalconvergence
\end{theorem}

Lastly, we consider the running time of our algorithms.  Both algorithms require an additional forward pass, whose time complexity is some multiple of that required to compute $f(\vx)$.  We state our result in terms of the number of multiplications, which is the dominant term in the time complexity for the loss functions of interest.

\newcommand{\thmtimecomplexity}{
Let $f: \reals^n \to \reals$ be a function such that for any $\vx \in \reals^n$, computing $f(\vx)$ requires $M_0$ multiplications, and calling $\mathrm{DirectionOracle}$ requires $M_1$ multiplications.  Then, using AutoBound as the PolynomialUpperBound function, each iteration of \saferate requires $O(M_1 + M_0 k \log(k))$ multiplications, where $k$ is the polynomial degree, and each iteration of \safecombination requires $O(M_1 + M_0 d^\omega)$ multiplications, where $\omega < 2.37286$ is the matrix multiplication exponent.
}
\newcommand{\thmtimecomplexityproof}{
On each iteration, \saferate must call $\mathrm{DirectionOracle}$ (using $M_1$ multiplications), and then must call AutoBound.  For each multiplication required to compute $f(\vx_t)$, AutoBound needs to multiply two degree $k$ interval polynomials, each of which has the property that only the final coefficient is an interval (and all other coefficients are scalars).  Using the Fast Fourier Transform, the product of two such interval polynomials can be computed using $O(k \log(k))$ multiplications.  Therefore, each iteration takes $O(M_1 + M_0 k \log(k))$ multiplications.

For \safecombination, the proof is similar, except that for each multiplication needed to compute $f(\vx_t)$ we must now multiply two multivariate quadratic interval polynomials.  Up to a constant factor, this has the same time complexity as multiplying two $d$ by $d$ matrices.  This can be done in time $O(d^\omega)$, where $\omega$ is the matrix multiplication exponent.  Thus the total time complexity is $O(M_1 + M_0 d^\omega)$.  At the time of writing, the best bound on the matrix multiplication exponent is $\omega < 2.37286$ \cite{alman2021refined}.
}
\begin{theorem} \label{thm:time_complexity}
\thmtimecomplexity
\end{theorem}

\section{Experiments} \label{sec:mm_experiments}

We now compare our universal MM optimizers to existing optimizers, on a suite of full-batch optimization problems.

The code for the AutoBound algorithm we use to derive majorizers is available on GitHub.\footnote{\url{https://github.com/google/autobound}}

\subsection{Optimizers} \label{sec:optimizers}

We compare \saferate and \safecombination to four existing optimizers: Adam \cite{kingma2015adam}, AdaGrad \cite{duchi2011adaptive}, gradient descent, and backtracking line search using the Armijo-Goldstein condition \cite{armijo1966minimization}.

For Adam, AdaGrad, and gradient descent, we consider all learning rates in the set $\set{10^i: i \in \integers, -4 \le i \le 1}$, and present results for the best-performing learning rate (learning rates that are powers of 10 outside this grid performed poorly on all problems).  Other Adam hyperparameters are left at the default values recommended in the original paper \cite{kingma2015adam}.

\emph{Backtracking line search} refers to a gradient descent optimizer that chooses the learning rate on each step using a backtracking line search first described by Armijo \cite{armijo1966minimization}.  Starting at a point $\vx$, backtracking line search seeks to find a point $\vx'$ that reduces the loss by at least $\frac \alpha 2 \norm{\nabla f(\vx)}^2$, where $\alpha$ starts at an initial value $\alpha_0 \in \reals$, and is halved until this condition is satisfied.  We consider all values of $\alpha_0 \in \set{10^i: i \in \integers, -5 \le i \le 3}$, and report results for the value of $\alpha_0$ that achieved the minimum loss.

These optimizers and hyperparameter grids are summarized in Table~\ref{tab:optimizers}.

\begin{table*}[h]
        \caption{Optimizers used in our experiments.}
        \label{tab:optimizers}
  \centering
        \begin{small}
        \begin{sc}
                \begin{tabular}{ll}
    \toprule
          Optimizer & Hyperparameters \\
    \midrule

\makecell[l]{AdaGrad \cite{duchi2011adaptive} \\ (diagonal matrix version) } & \makecell[l]{$\eta \in \set{10^i: i \in \set{10^i: i \in \integers, -4 \le i \le 1}}$, \\ $\delta = 0$} \\
Adam \cite{kingma2015adam} & \makecell[l]{$\eta \in \set{10^i: i \in \set{10^i: i \in \integers, -4 \le i \le 1}}$, \\ $\beta_1 = 0.9$, $\beta_2 = 0.999$, $\epsilon=10^{-8}$} \\
Backtracking Line Search \cite{armijo1966minimization} & $\alpha_0 \in \set{10^i: i \in \integers, -5 \le i \le 3}$ \\
GD & $\eta \in \set{10^i: i \in \set{10^i: i \in \integers, -4 \le i \le 1}}$ \\
SafeRate[AdaGrad] (ours)  & $\delta = 0$ \\
SafeRate[Adam] (ours)  & $\beta_1 = 0.9$, $\beta_2 = 0.999$, $\epsilon=10^{-8}$ \\
SafeRate[GD] (ours)  & -- \\
SafeCombination[per-layer AdaGrad] (ours) & $\delta = 0$ \\
SafeCombination[per-layer Adam] (ours)  & $\beta_1 = 0.9$, $\beta_2 = 0.999$, $\epsilon=10^{-8}$ \\
SafeCombination[per-layer GD] (ours) & -- \\
\bottomrule
  \end{tabular}
  \end{sc}
  \end{small}
\end{table*}

In addition, we evaluate the performance of \saferate using three choices for the directional oracle: SafeRate[GD] uses the negative gradient direction, SafeRate[AdaGrad] uses the AdaGrad direction (determined based on the observed history of gradient vectors), and SafeRate[Adam] uses the Adam direction.  The SafeRate[AdaGrad] and SafeRate[Adam] update directions are based on a learning rate of .1 for the underlying AdaGrad/Adam algorithm, and the recommended default values for all other hyperparameters.  Note that the learning rate affects only the scale of the update direction, and the behavior of \saferate is largely invariant to this scale (it is not strictly invariant because of the way the scale interacts with the trust region size).  For this reason we do not need to tune the learning rate.

Finally, we evaluate the performance of \safecombination, using three choices for the matrix of update directions.  For SafeCombination[per-layer GD], there is an update direction for each tensor $\mW$ that appears in the loss, and the update direction for a tensor $\mW$ equals the negative gradient with entries for tensors other than $\mW$ zeroed out.  Thus, when applied to a neural network optimization problem, SafeCombination[per-layer GD] computes a per-layer learning rate (hence its name).  SafeCombination[per-layer AdaGrad] and SafeCombination[per-layer Adam] are similar, except they use the per-layer update directions given by AdaGrad and Adam, respectively, rather than the negative gradient.  The Adam/AdaGrad update directions are determined in the same way described in the previous paragraph.

When plotting an optimizer's performance as a function of the number of steps, we consider each distinct point evaluated as a separate step (for backtracking line search, this includes points where the sufficient-loss-decrease condition is not satisfied).

\subsection{One-Dimensional Problems} \label{sec:one_d_problems}

We begin by experimenting with synthetic one-dimensional problems.  This will allow us to understand the behavior of \saferate qualitatively, in a low-dimensional setting where the results can be easily visualized.

\subsubsection{Losses}

\begin{table*}[h]
        \caption{One-dimensional loss functions used in our experiments.}
        \label{tab:one_d_losses}
  \centering
        \begin{small}
        \begin{sc}
                \begin{tabular}{lll}
    \toprule
          Problem name & Loss function \\
    \midrule

  1-d least squares linear regression &  $x \mapsto (x - \frac 3 2)^2$ \\
  1-d linear regression with non-normal errors & $x \mapsto (x - 3)^4$ \\
  1-d logistic regression & $x \mapsto \frac 2 3 \log(1+\exp(x)) + \frac 1 3 \log(1+\exp(-x))$ \\
  Optimizing a single neural network parameter & $x \mapsto ((\mathrm{sigmoid}(x - 10) - \frac 1 2)^2$ \\
\bottomrule
  \end{tabular}
  \end{sc}
  \end{small}
\end{table*}

Table~\ref{tab:one_d_losses} summarizes the loss functions used in these experiments.  Each loss is a one-dimensional example of a widely-studied problem.  For least squares linear regression, we consider a problem with a single feature, making the loss a quadratic polynomial.  The specific choice of quadratic does not qualitatively change the results; we choose $(x - \frac 3 2)^2$.  We also consider linear regression with error terms drawn from a \emph{generalized symmetric normal distribution} \cite{mineo2005software,zeckhauser1970linear}.  Setting the $\beta$ parameter of the generalized symmetric normal distribution to 4 results in a quartic polynomial loss, where again the specific choice of quartic does not qualitatively change the results.  For logistic regression, we choose a one-dimensional problem where two thirds of the labels are negative, and there is a single feature whose value is always 1.

As a one-dimensional example of neural network training, we consider the loss $f(x) = ((\mathrm{sigmoid}(x - 10) - \frac 1 2)^2$.  Minimizing this loss can be thought of as optimizing the bias parameter in the first layer of a single-hidden-layer neural network with sigmoid activation functions and squared error, with a training set of size 1 and appropriate initial values for parameters other than $x$.  This loss is interesting because initially (when $x = 0$) the input to the sigmoid is very small, resulting in very small initial gradients that can pose a challenge for traditional optimizers.

\begin{figure}[h]
\begin{center}
\includegraphics[width=0.45\linewidth]{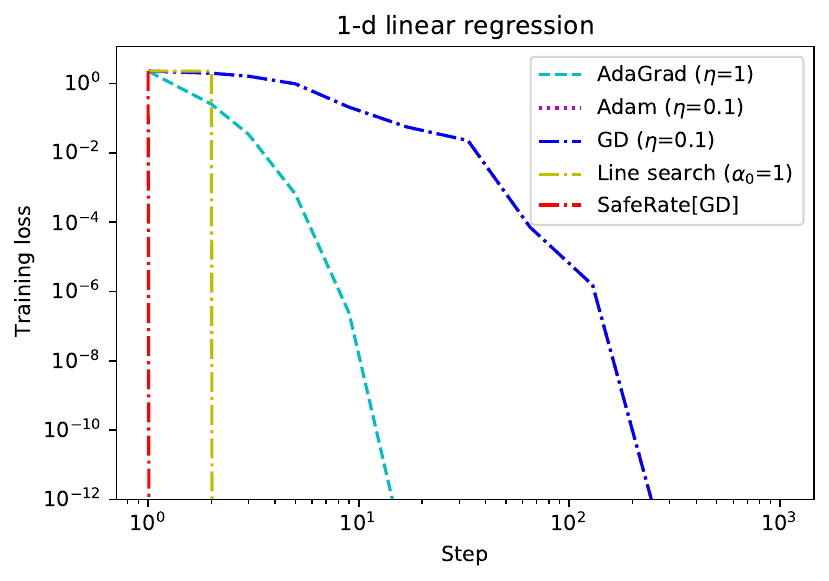}
\includegraphics[width=0.45\linewidth]{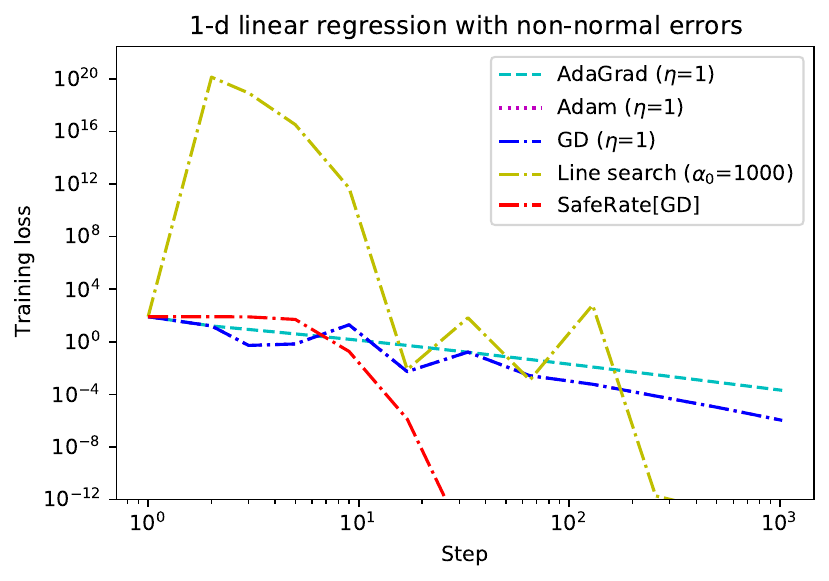}
\includegraphics[width=0.45\linewidth]{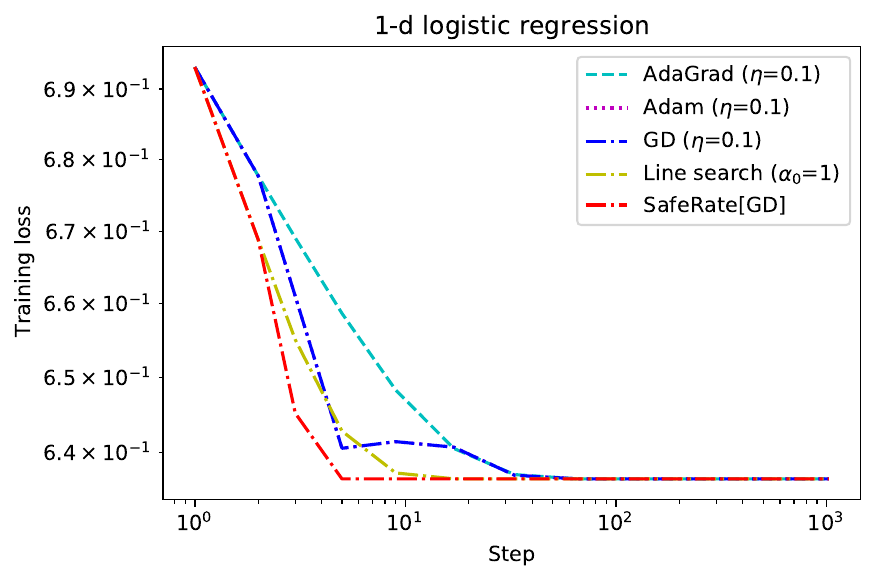}
\includegraphics[width=0.45\linewidth]{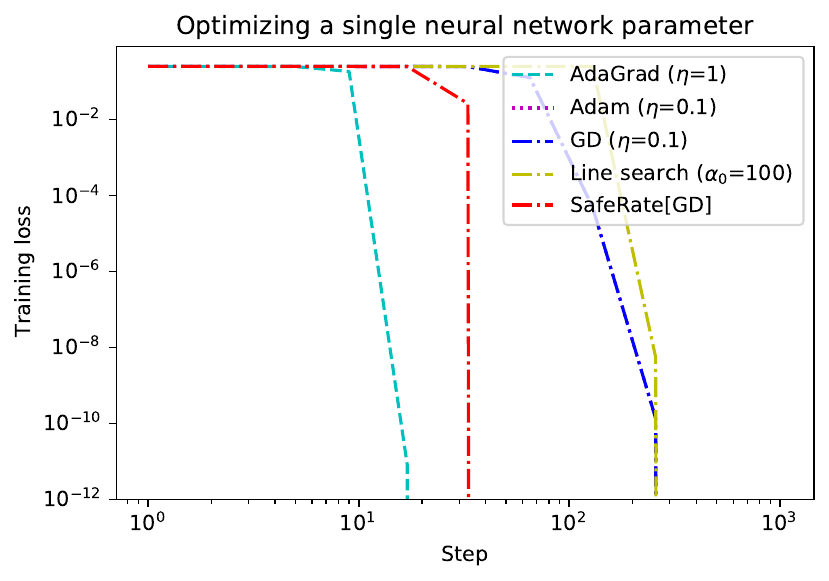}
\caption{Comparison of optimizers on one-dimensional optimization problems.  Each plot shows the loss as a function of the number of iterations (log scale).  For optimizers other than \saferate, the plot shows the best-performing hyperparameter from the grid defined in Table~\ref{tab:optimizers} (the best-performing hyperparameter is different for each plot).  Despite not requiring hyperparameter tuning, \saferate generally outperforms tuned versions of all baseline optimizers.  For these one-dimensional problems, all optimizers require similar wall time per step.}
\label{fig:one_d_optimizer_comparison}
\end{center}
\end{figure}

\subsubsection{Performance comparison}

Figure~\ref{fig:one_d_optimizer_comparison} compares the performance of \saferate to that of tuned versions of the baseline optimizers, on the four one-dimensional optimization problems given in Table~\ref{tab:one_d_losses}.  For optimizers other than \saferate (which has no hyperparameters), each plot shows the best-performing\footnote{Performance is measured by the minimum loss reached on any step.} hyperparameter settings for each problem, considering all the settings in the grid defined by Table~\ref{tab:optimizers}.  For these one-dimensional problems, all optimizers require similar wall time per step.

We make the following observations:
\begin{itemize}
    \item Despite not requiring hyperparameter tuning, \saferate outperforms tuned version of all baseline optimizers on all four problems, with one exception (with a learning rate of 1, AdaGrad performs better on the neural network parameter optimization problem). 
	\item For 1-d least squares linear regression, the quadratic majorizer derived using \autobound is exact, and thus \saferate jumps to the global minimum on the first step.%
	\item For 1-d linear regression with non-normal errors, \saferate converges super-linearly\footnote{An optimizer is said to converge linearly if the log of the optimality gap decreases linearly as a function of the log of the number of steps.}, whereas gradient descent and AdaGrad appear to converge linearly.	
\end{itemize}

\subsubsection{Safe learning rates}

We now examine in more detail how \saferate behaves on these one-dimensional problems.  Recall that the ``safe'' learning rate computed by \saferate on step $t$ depends on two things: the current iterate $x_t$, and the maximum learning rate $\etamax_t$ (which determines the trust region, $[0, \etamax_t]$).  Figure~\ref{fig:one_d_safe_etas} depicts the safe learning rate as a function of $x = x_t$, for various values of the trust region.

\begin{figure}[h]
\begin{center}
\includegraphics[width=0.45\linewidth]{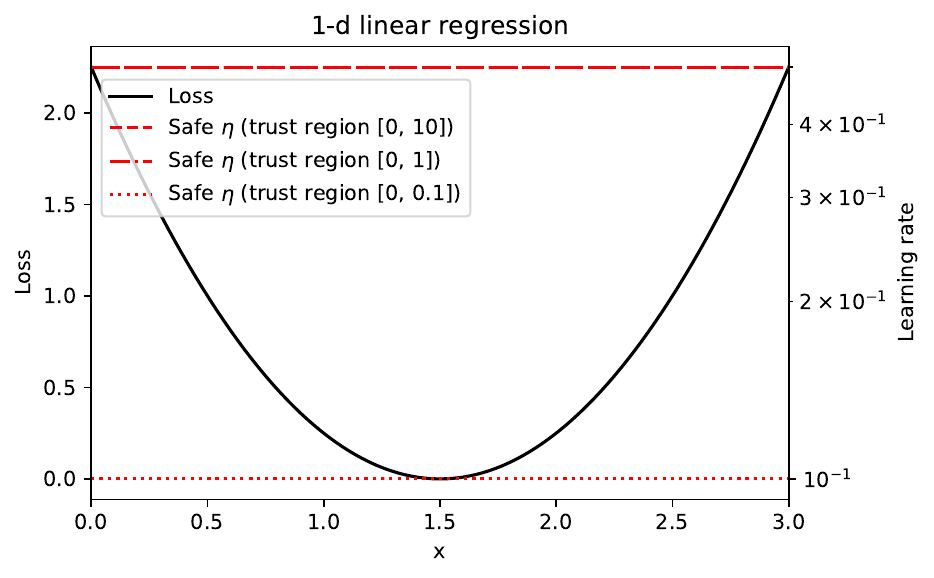}
\includegraphics[width=0.45\linewidth]{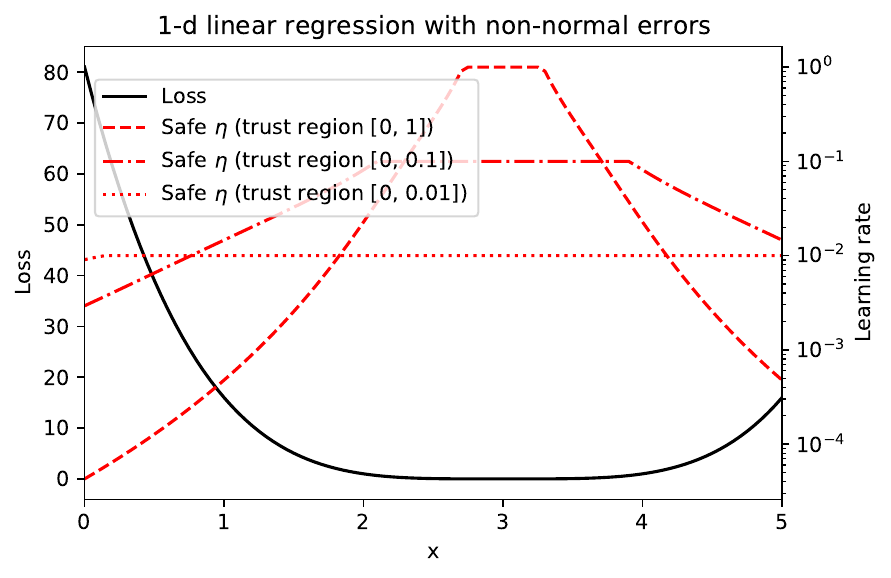}
\includegraphics[width=0.45\linewidth]{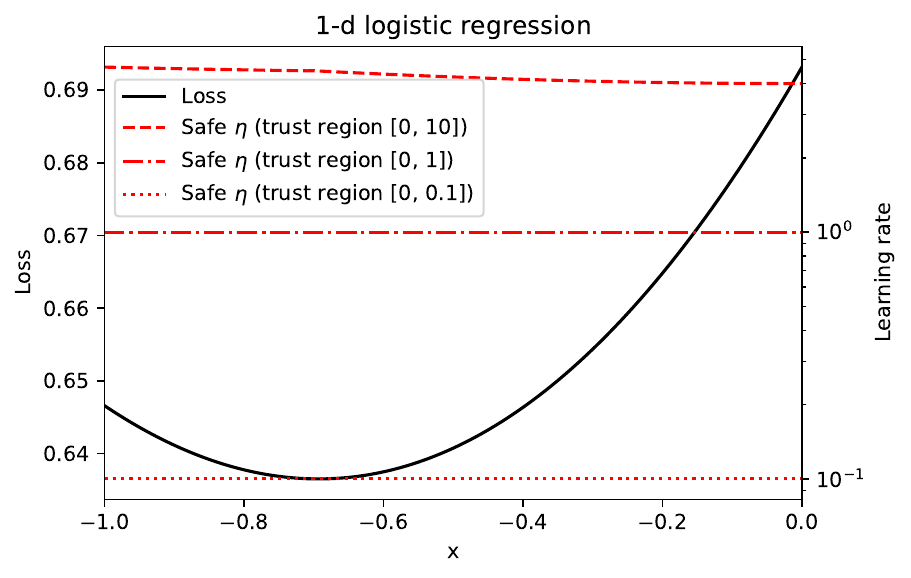}
\includegraphics[width=0.45\linewidth]{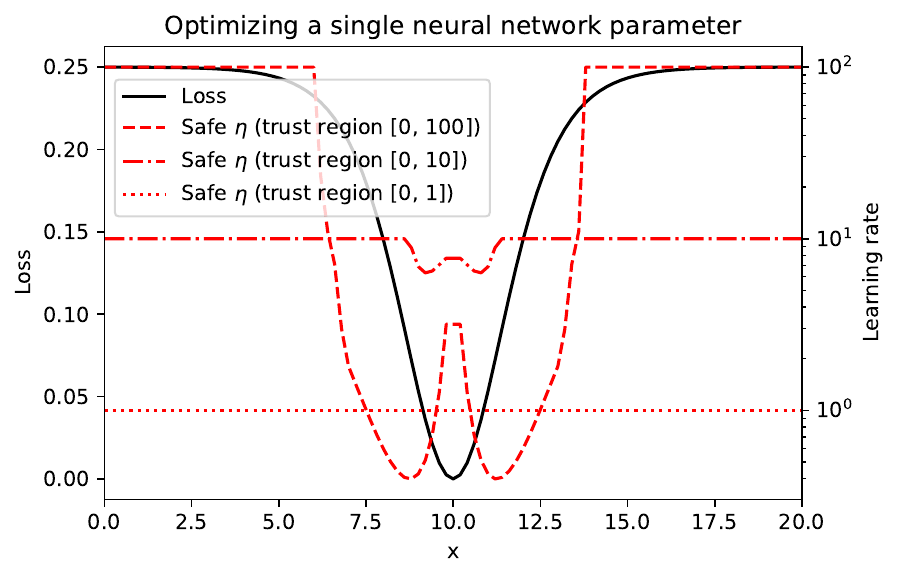}
\caption{Safe gradient descent learning rates for various one-dimensional problems, as a function of the current iterate $x$ and the trust region.  For each $x$, each plotted learning rate $\eta$ is computed using the symbolic expression for the loss $f$, and is guaranteed to reduce the loss:  $f(x - \eta \nabla f(x)) \le f(x)$.}
\label{fig:one_d_safe_etas}
\end{center}
\end{figure}

In examining Figure~\ref{fig:one_d_safe_etas}, several things are worth noting:
\begin{itemize}
  \item In general, the learning rate depends in a non-trivial way on the current iterate $x$, and can be orders of magnitude larger for some $x$ than for others.
  \item For some problems, the ``safe'' learning rate increases dramatically as one approaches the global minimum, while for other problems it decreases dramatically.
  \item The optimal trust region width (i.e., the one that lets us compute the largest ``safe'' learning rate) also depends on $x$.  For example, for the linear regression problem with non-normal errors, at $x = 0$ the optimal trust region width is .01; at $x = 1$ it is 0.1, and at $x = 2.9$ it is 1.
\end{itemize}
These points illustrate the potential for \saferate to non-trivially adapt the learning rate during the course of optimization, and make clear that the trust region must adapt over time if we wish to compute the largest possible safe learning rates.

\subsubsection{SafeRate trust region adaptation}

Figure~\ref{fig:one_d_saferate_etas} shows how the learning rate $\eta_t$ and the maximum learning rate $\etamax_t$ (which determines the trust region) evolve when running \saferate on each of the one-dimensional problems.

\begin{figure}[h]
\begin{center}
\includegraphics[width=0.45\linewidth]{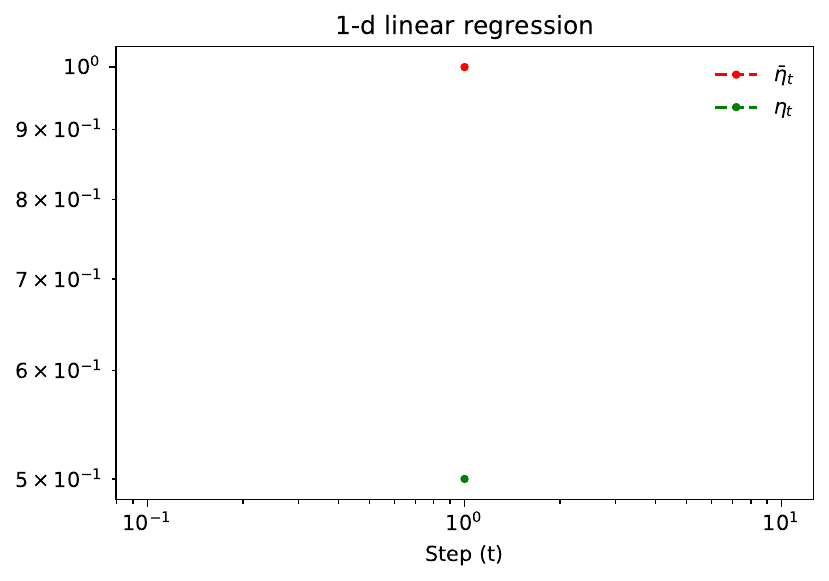}
\includegraphics[width=0.45\linewidth]{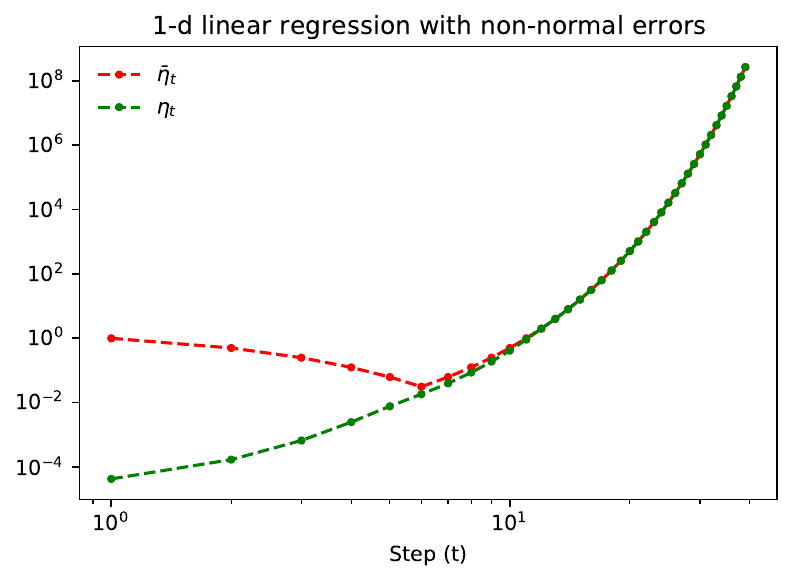}
\includegraphics[width=0.45\linewidth]{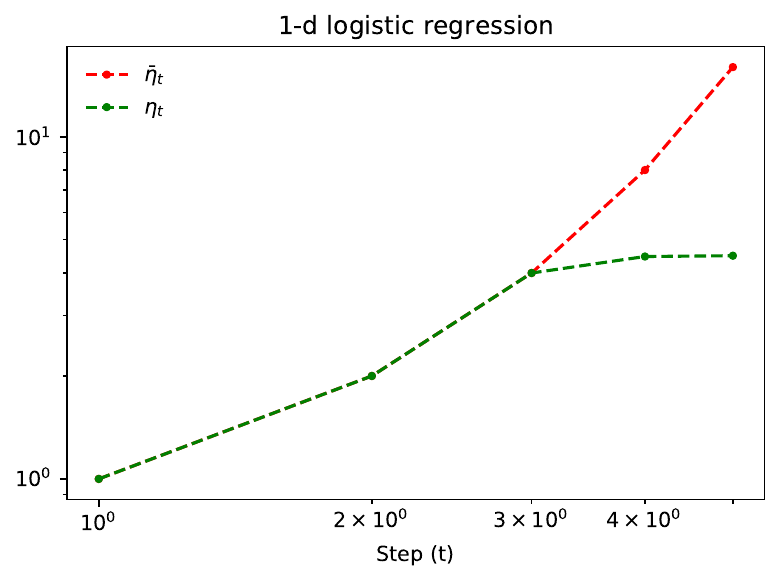}
\includegraphics[width=0.45\linewidth]{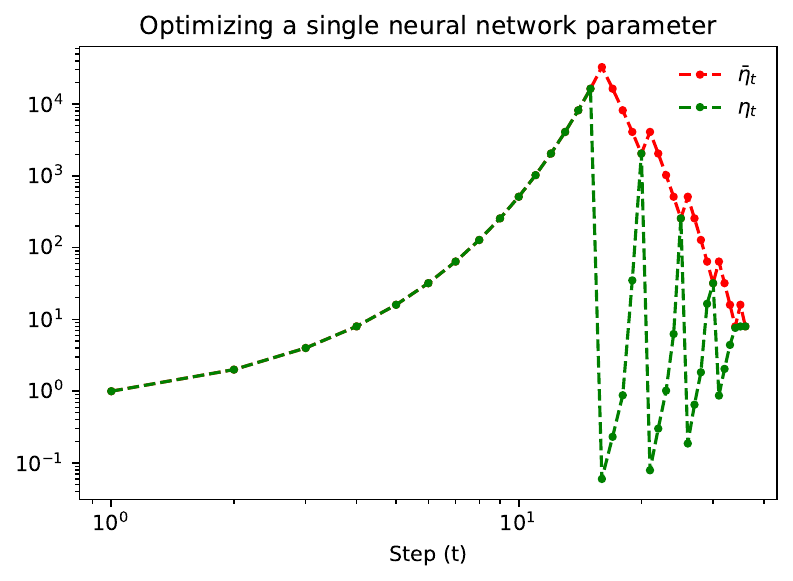}
\caption{The evolution of $\eta_t$ and $\etamax_t$ when using \saferate to minimize various one-dimensional losses.}
\label{fig:one_d_saferate_etas}
\end{center}
\end{figure}

We summarize the results shown in Figure~\ref{fig:one_d_saferate_etas} as follows:
\begin{itemize}
  \item For least squares linear regression, \saferate converges in one step, and thus trust region adaptation plays no role.
  \item For linear regression with non-normal errors, the maximum learning rate $\etamax_t$ is initially too large, and hence $\eta_t \ll \etamax_t$, leading to slow progress.  As a result, \saferate decreases $\etamax_t$ exponentially, which causes $\eta_t$ to increase exponentially.  Then, once $\eta_t \approx \etamax_t$, both $\etamax_t$ and $\eta_t$ \emph{increase} exponentially for the remainder of the optimization process.  We note that this learning rate schedule is very different from the regret-bound-minimizing schedule used in algorithms such as AdaGrad \cite{duchi2011adaptive} and FTPRL \cite{mcmahan2010adaptive}, which always decreases the learning rate and does so at a polynomial rate.
  \item For logistic regression, $\etamax_t$ is initially too small, causing $\eta_t$ to be capped at its maximum value.  This causes $\etamax_t$ to double until this is no longer the case, leading to convergence in a few steps.
  \item For the one-dimensional neural network problem, the gradients are initially very small, and a very large learning rate is necessary to make progress.  As in the logistic regression problem, $\eta_t$ is initially capped at its maximum value, causing $\etamax_t$ to double until this is no longer the case.  Then, once the optimizer reaches the part of the sigmoid curve where the gradients are larger, $\eta_t$ suddenly decreases by more than four orders of magnitude.  This causes $\etamax_t$ to decrease, causing $\eta_t$ to increase again, and $\eta_t$ continues to oscillate up and down for the remainder of the optimization run, while at the same time the loss decreases rapidly (as shown in Figure~\ref{fig:one_d_optimizer_comparison}).
\end{itemize}
Overall, the simple doubling/halving heuristic used by \saferate is very effective on these problems, and leads to qualitatively different behaviors on each problem.

\subsection{Random Regression Problems}

For our next experiment, we evaluate \saferate on randomly-generated linear and logistic regression problems.  For linear regression, we consider both least-squares linear regression, as well as linear regression with non-normal error terms, described in more detail below.

\newcommand{\betastar}{\boldsymbol{\beta}^*}

We generate random, well-specified regression problems by sampling feature vectors from a normal distribution with a specified covariance matrix, and generating labels based on a known model.  Letting $d$ denote the number of features, the covariance matrix is $\mZ^\tee \mZ$, where $\mZ$ is a $d$ by $d$ matrix whose elements are drawn from a standard normal distribution.  The true model is $\betastar \in \reals^d$, where each coordinate of $\betastar$ is drawn independently from a standard normal distribution.  The labels are generated as follows:
\begin{itemize}
  \item For least-squares linear regression, the label for an example with feature vector $\va$ is drawn from a normal distribution with mean $\va^\tee \betastar$ and standard deviation 0.1.
  \item For logistic regression, the label for an example with feature vector $\va$ is drawn from a Bernoulli distribution with mean $\frac{1}{1+\exp( - \va^\tee \betastar )}$.
  \item For linear regression with non-normal errors \cite{mineo2005software,zeckhauser1970linear}, the label for an example with feature vector $\va$ is drawn from a generalized symmetric normal distribution with parameters $\alpha = .1$ and $\beta = 4$.
\end{itemize}
For all three problems, the loss is the negative log-likelihood.  We use $d = 100$ features and $n = 10000$ training examples.

For linear regression with non-normal errors, the loss is $f(\vx) = \sum_{i=1}^n (\mA_i^\tee \vx - \vb_i)^\beta$, where $\mA$ is the feature matrix and $\vb$ is the label vector.  Because we set $\beta = 4$, the loss grows as a quartic function of the estimation error (as in the one-dimensional example in \S\ref{sec:one_d_problems}).

It is worth noting that, for least squares linear regression, the quadratic upper bound that \saferate uses to compute a safe learning rate is tight, and the \saferate learning rate is therefore the rate that maximally reduces the loss on each step.  Thus, applied to least squares linear regression, SafeRate[Adam] can be thought of as a variant of Adam that at each step moves as far as possible in the Adam direction, stopping at the point where further movement would increase the loss.  In contrast, when applied to logistic regression or linear regression with non-normal errors, \saferate will always choose a learning rate that reduces the loss, but this rate will be smaller than the one that reduces the loss maximally.

\ignore{
Least squares linear regression corresponds to maximum likelihood estimation under the assumption of normally-distributed error terms.  Although this is by far the most commonly-used form of linear regression in practice, other distributions for the error terms are possible.  In particular, several authors have considered linear regression with error terms distributed according to a \emph{generalized symmetric normal distribution}  (a.k.a.\ exponential power distribution) \cite{mineo2005software,zeckhauser1970linear}, so that in one dimension, $b_i = a_i x + \epsilon_i$ with probability density:
\begin{equation}
  g(\epsilon_i) \propto \exp \paren {- \frac {| \epsilon_i |^\beta} {\alpha^\beta} }
\end{equation}
for $\alpha, \beta \in \reals$.
For fixed $\alpha$ and $\beta$, the negative log-likelihood therefore has the form:
\begin{equation}
  f(x) = C + \sum_{i=1}^n |a_i x - b_i|^\beta
\end{equation}
where $C \in \reals$ is a constant whose value does not affect the optimization problem.

The case $\beta > 2$ is interesting from an optimization theory perspective, because although $f$ is convex, it grows faster than quadratically, which can pose a challenge for optimization algorithms whose theory assumes sub-quadratic losses.  
}

Figure~\ref{fig:reg_comparison} compares SafeRate[GD], SafeRate[AdaGrad], and SafeRate[Adam] to tuned versions of the baseline optimizers given in Table~\ref{tab:optimizers}.  As in our previous experiment, we show only the best-performing hyperparameter value for each baseline optimizer.  The plots on the left use the number of steps as the horizontal axis, while the plots on the right use wall time.

Examining Figure~\ref{fig:reg_comparison}, we note that:
\begin{itemize}
  \item For all three problems, SafeRate[Adam] outperforms all other optimizers in terms of the loss reached after a given number of steps.  For linear regression with non-normal errors, it is also better in terms of wall time, while for the other two problems it is slightly worse in terms of wall time.
  \item For the two linear regression problems, SafeRate[GD] outperforms the best fixed GD learning rate by a wide margin.
  \item SafeRate[AdaGrad] consistently outperforms all fixed AdaGrad learning rates early in optimization (see also Figure~\ref{fig:reg_saferate_vs_grid}), but a well-tuned AdaGrad learning rate performs slightly better asymptotically.   For linear regression, this implies that greedily choosing the learning rate that maximally reduces the loss (as SafeRate[AdaGrad] does) does not produce the best loss asymptotically (although it comes close).
\end{itemize}

For all three problems, \saferate requires three matrix-vector products per step, whereas each step of GD, Adam, and AdaGrad requires only two.  For this reason, we would expect SafeRate[Adam] to take about 1.5 times as long as Adam to complete a fixed number of steps.  Empirically, however, it takes roughly twice as long, which may point to suboptimality of the generated computation graph for \saferate (and an opportunity to improve the results).

\begin{figure}[H]
\begin{center}
\includegraphics[width=0.9\linewidth]{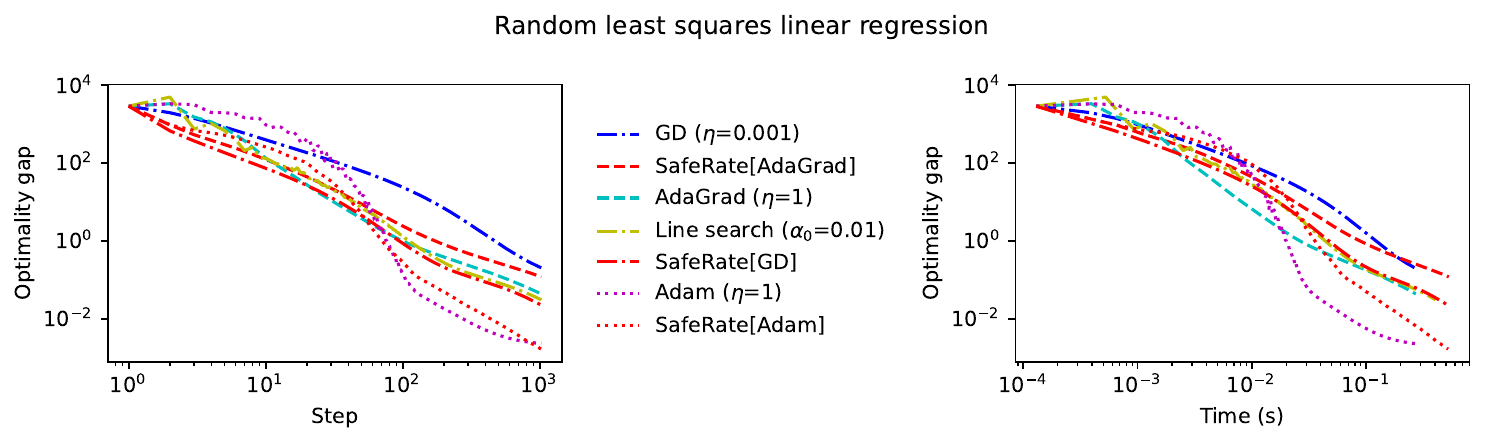}
\\ \vspace{.4cm}
\includegraphics[width=0.9\linewidth]{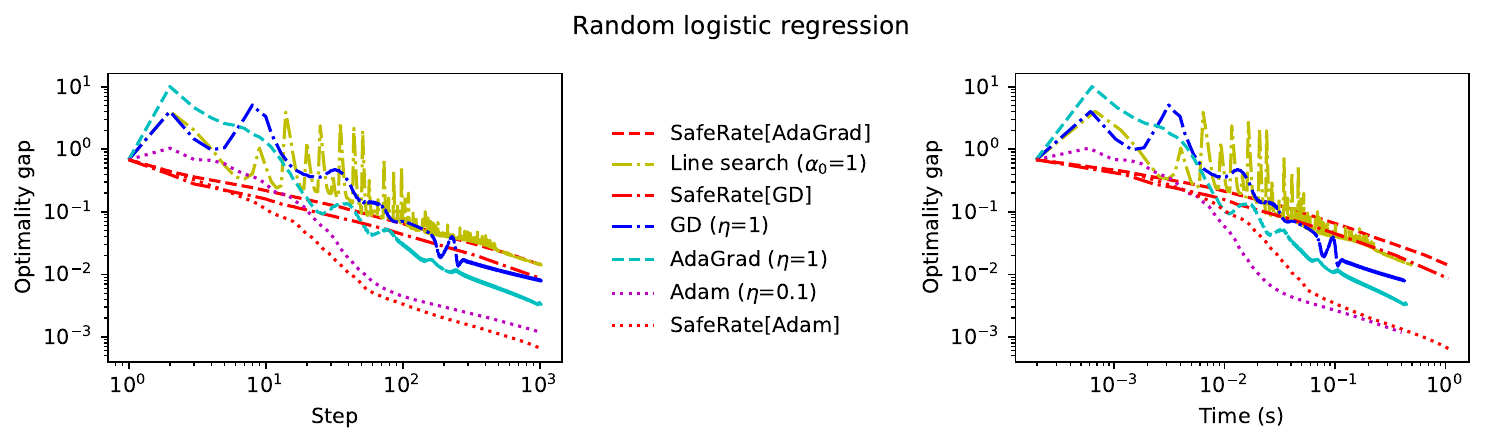}
\\ \vspace{.4cm}
\includegraphics[width=0.9\linewidth]{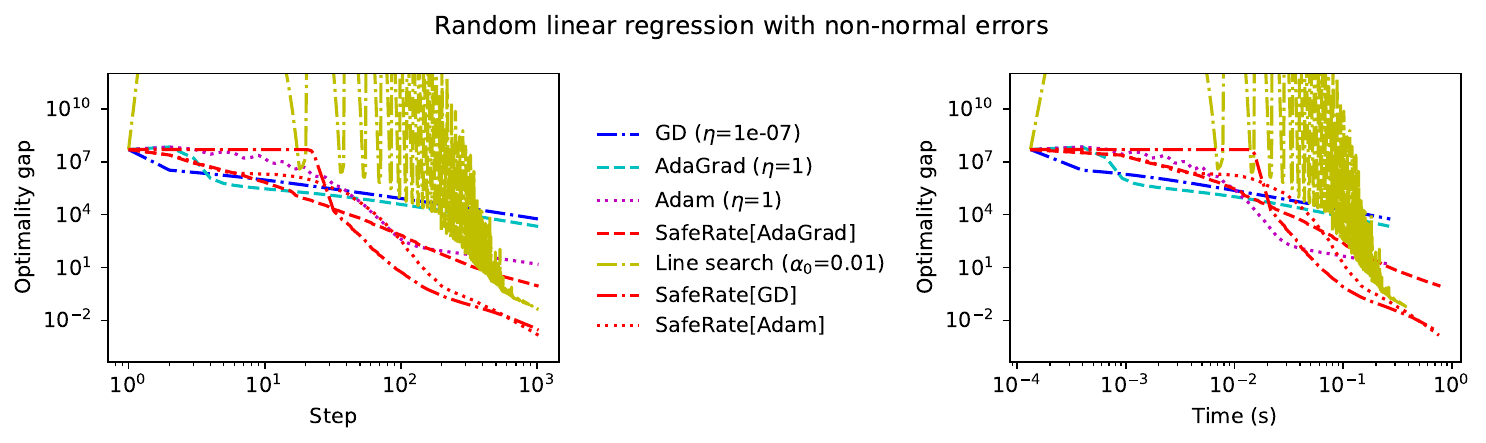}
\end{center}
\caption{Comparison of optimizers on randomly-generated linear and logistic regression problems.  Each plot shows the loss as a function of the number of iterations (left) or wall time (right), on a log scale.  For optimizers other than \saferate, the plot shows the best-performing hyperparameter from the grid defined in Table~\ref{tab:optimizers}.
Despite not requiring hyperparameter tuning, SafeRate typically performs about as well as the best learning rate in our grid, and sometimes outperforms it dramatically.}
\label{fig:reg_comparison}
\end{figure}

\begin{figure}[H]
\begin{center}
\includegraphics[width=0.3\linewidth]{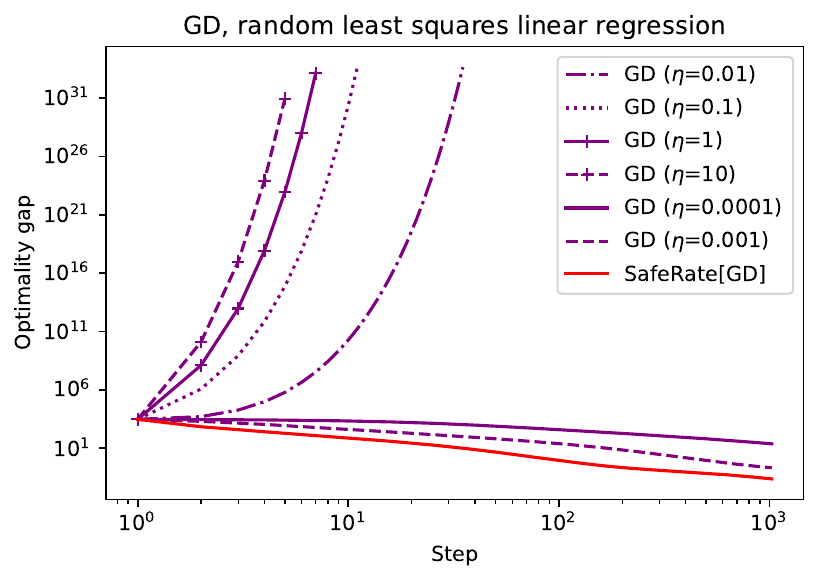}
\includegraphics[width=0.3\linewidth]{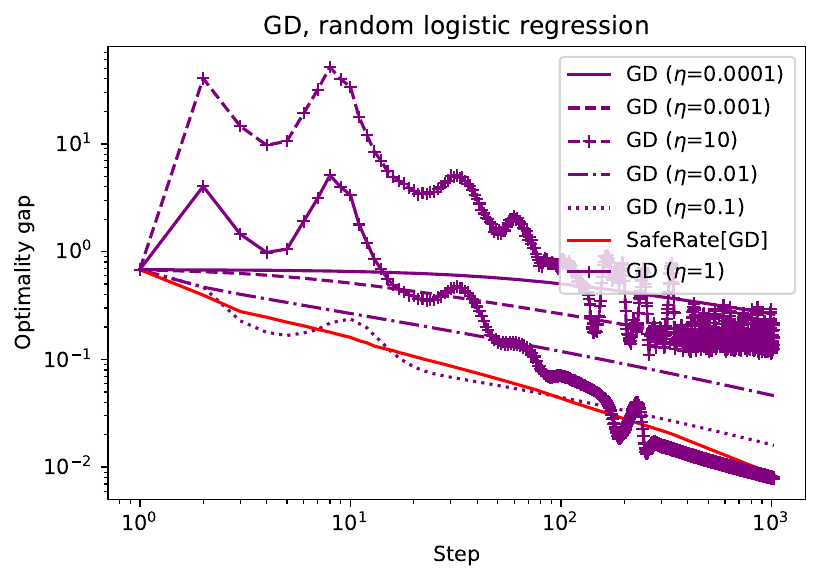}
\includegraphics[width=0.3\linewidth]{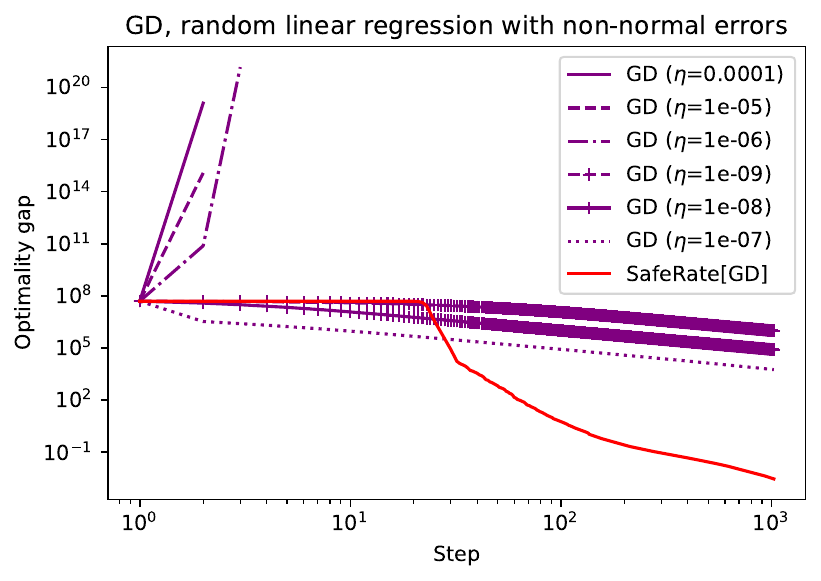}

\includegraphics[width=0.3\linewidth]{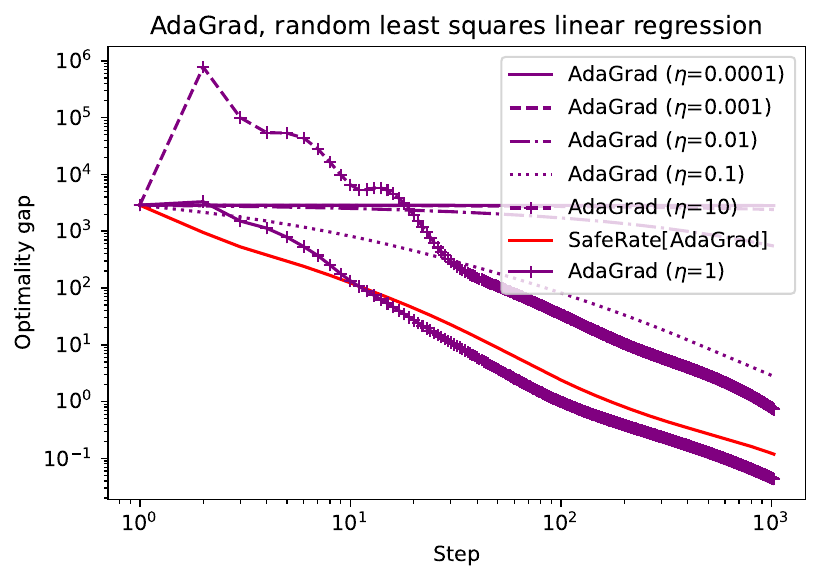}
\includegraphics[width=0.3\linewidth]{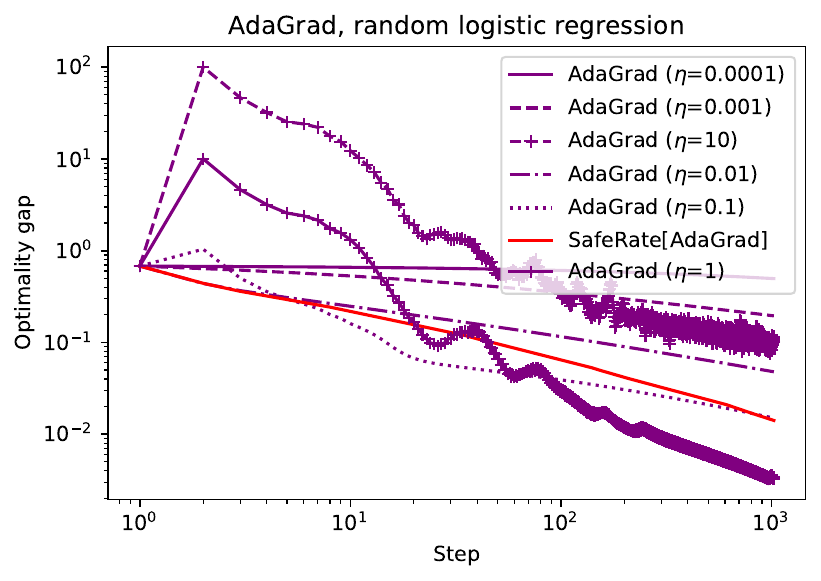}
\includegraphics[width=0.3\linewidth]{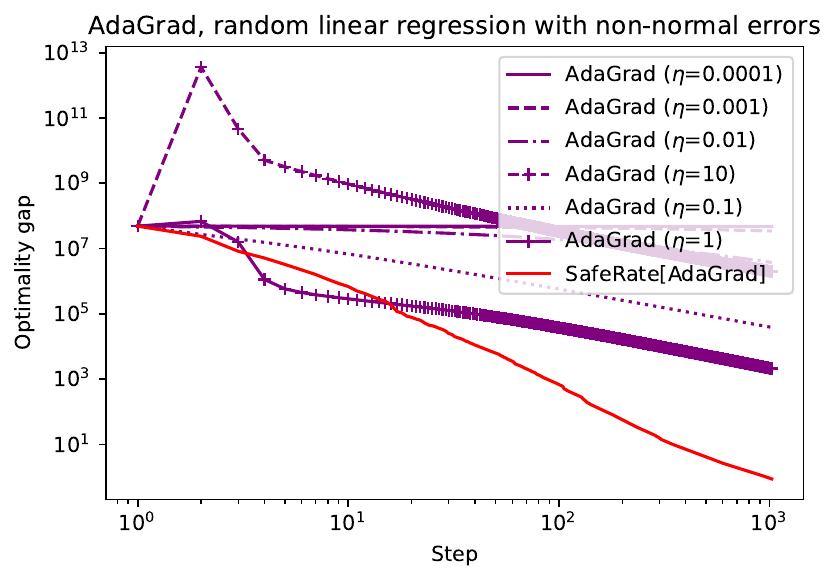}

\includegraphics[width=0.3\linewidth]{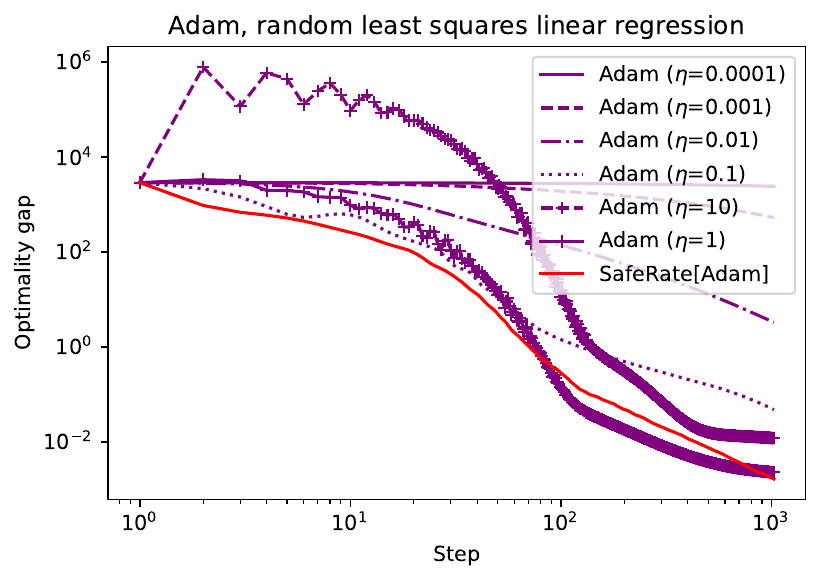}
\includegraphics[width=0.3\linewidth]{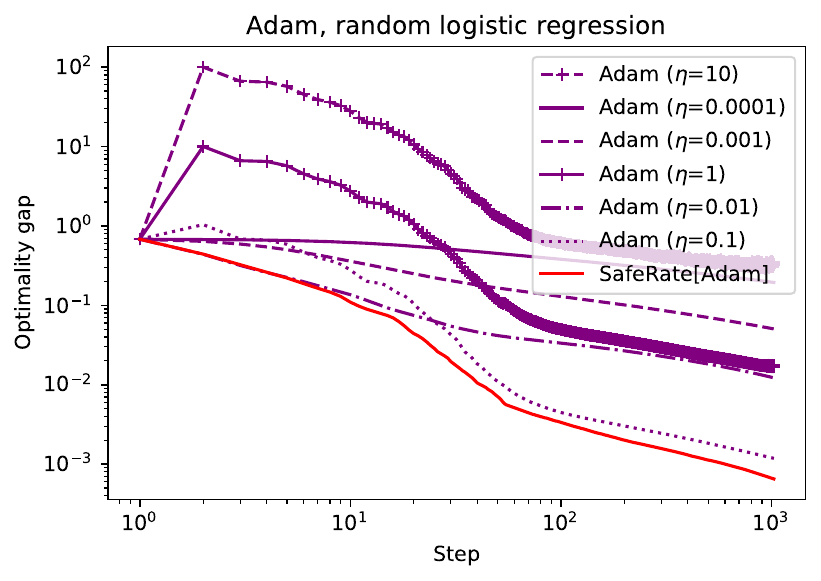}
\includegraphics[width=0.3\linewidth]{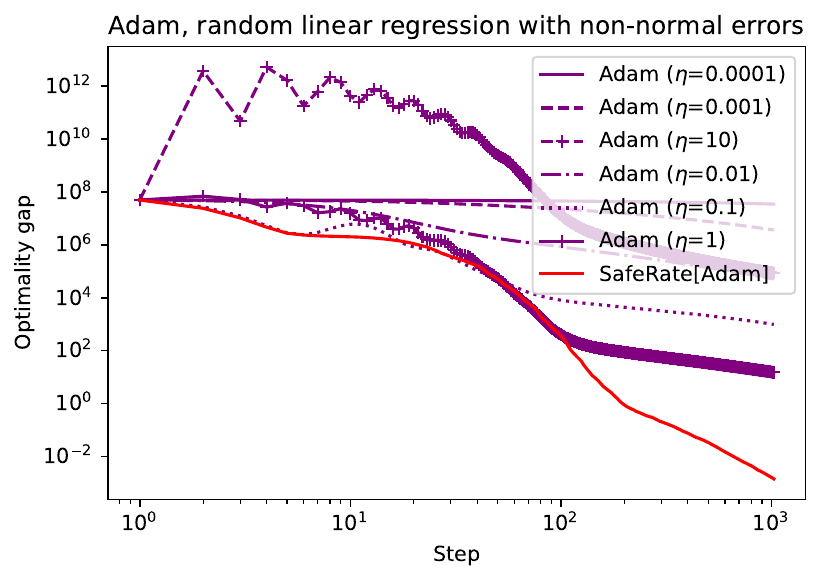}

\caption{Comparison of \saferate to optimizers with a learning rate hyperparameter, when used to solve randomly-generated linear and logistic regression problems.  
\emph{On these problems, SafeRate takes roughly twice as much wall time per step as the other algorithms (see Figure~\ref{fig:reg_comparison}).}}
\label{fig:reg_saferate_vs_grid}
\end{center}
\end{figure}

Figure~\ref{fig:reg_saferate_vs_grid} presents a more detailed view of these results.  In this figure, there are separate plots for GD, AdaGrad, and Adam that include results for all learning rates in the grid, plus results for the corresponding SafeRate algorithm.  Figure~\ref{fig:reg_saferate_vs_grid} makes it clear that, in addition to being competitive with the best learning rate in the grid, \saferate dramatically outperforms suboptimal learning rates, some of which lead to divergence or very slow progress.

\subsection{Multi-Layer Perceptrons} \label{sec:mlps}

For our final set of experiments, we use \saferate and \safecombination to train deep networks to classify images.

Specifically, we train a fully connected network with one or two hidden layers to classify images from the MNIST dataset.  We use 1000 hidden units, the $\mathrm{softplus}$ activation function, and the square loss.  Following the recommendation of \cite{hui2021evaluation}, we do \emph{not} use a final softmax layer when computing the square loss.  We train in the full-batch setting, using the first 1000 images as our training set.  To work around numerical issues in the \autobound algorithm, we use {\tt  float64} for the two-hidden-layer experiments with \saferate and \safecombination, and to keep the comparison fair we also use {\tt  float64} for the baseline optimizers.

As in our previous experiment, we evaluate \saferate using three different choices of update direction, namely the directions given by GD, AdaGrad, and Adam (based on the observed sequence of gradients, as described in \S\ref{sec:optimizers}).  Additionally, we evaluate SafeCombination[per-layer GD], SafeCombination[per-layer AdaGrad], and SafeCombination[per-layer Adam], which compute adaptive per-layer learning rates using the update directions given by GD, AdaGrad, and Adam, respectively.

\begin{figure}[H]
\begin{center}
\includegraphics[width=0.9\linewidth]{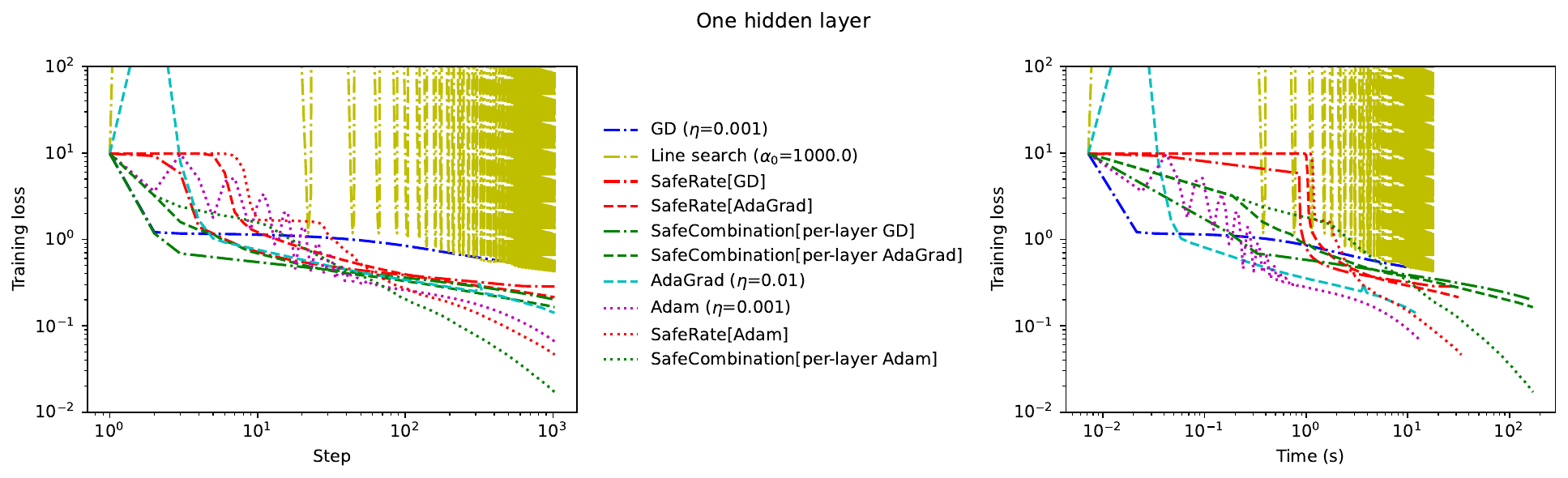}
\\ \vspace{.4cm}
\includegraphics[width=0.9\linewidth]{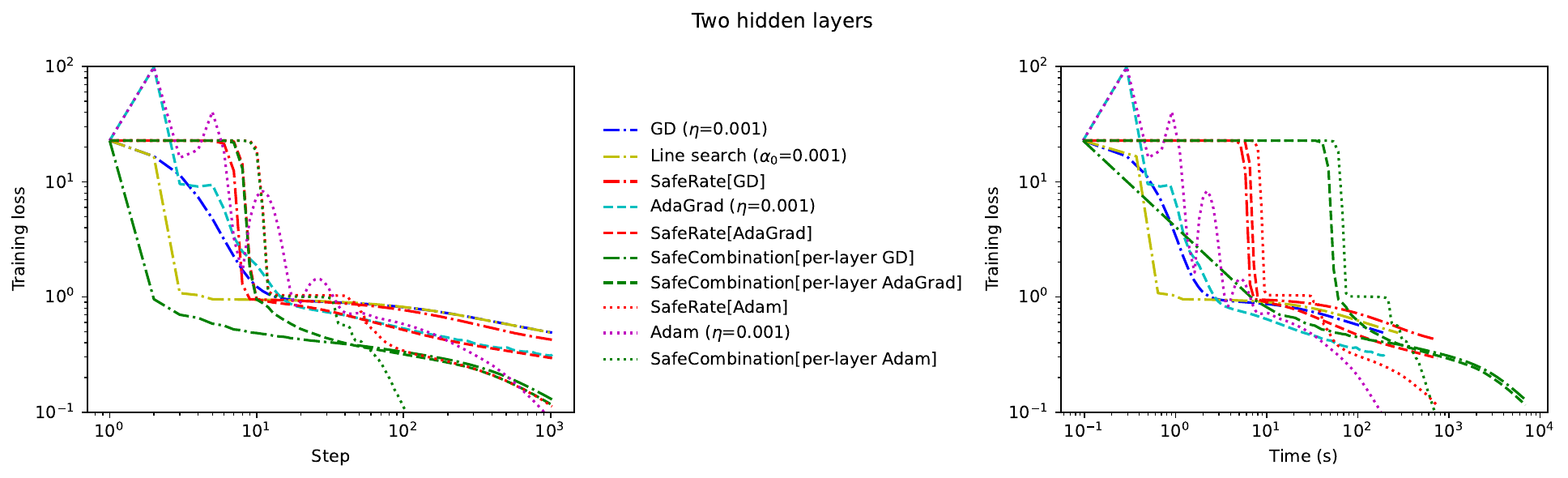}
\end{center}
\caption{Comparison of optimizers, training a multi-layer perceptron on a subset of the MNIST dataset.  Each plot shows the loss as a function of the number of iterations (left) or wall time (right), on a log scale.  For optimizers other than \saferate and \safecombination, the plot shows the best-performing hyperparameter from the grid defined in Table~\ref{tab:optimizers}.}
\label{fig:mnist_comparison}
\end{figure}

Figure~\ref{fig:mnist_comparison} shows the training loss reached by each optimizer as a function of the number of iterations, and as a function of wall time.  We note that:
\begin{itemize}
  \item For the one-hidden-layer problem, SafeRate[Adam] reaches lower training loss than all the baseline optimizers (after 1024 steps).  However, it requires about 2.5x as much wall time per step as Adam, and thus is somewhat worse in terms of training loss vs.\ wall time.
  \item SafeCombination consistently makes more progress per step than SafeRate, at the cost of additional computation.
  \item Using our current implementation, \ref{alg:safecombination} is very slow, requiring roughly 13x as much wall time per step as Adam on the one-hidden-layer problem, and roughly 100x as much time per step on the two-hidden-layer problem.\footnote{Applied to a neural network with $H$ hidden layers, \ref{alg:safecombination} requires $O(d^3 H)$ time per step, where $d$ is the number of update directions.  For SafeCombination[per-layer-Adam], we have $d = H + 1$, and thus the time per step is $O(H^4)$.}  Thus, \ref{alg:safecombination} is not currently competitive with the baseline optimizers in terms of wall time.  These wall time numbers should be taken with a grain of salt, however, as we suspect they could be significantly reduced by optimizing our implementation.\footnote{In particular, the wall time can likely be improved by taking advantage of the sparsity of the update directions.}
\end{itemize}

Figure~\ref{fig:mnist_saferate_vs_grid} presents a more detailed view of the same results, with separate plots for GD, Adam, and AdaGrad, showing the performance of various learning rates along with that of the corresponding SafeRate and SafeCombination algorithms.  To reduce clutter, only learning rates between $10^{-4}$ and $0.1$ are shown; learning rates outside this range performed poorly on both problems.  For the one-hidden-layer network, \saferate is generally able to offer performance comparable to the best learning rate in our grid on a per-step basis.  The results are best for gradient descent, where SafeRate[GD] takes an order of magnitude fewer steps to reach the loss that GD reaches after 1024 steps with a tuned learning rate.

\begin{figure}[H]
\begin{center}
\includegraphics[width=0.4\linewidth]{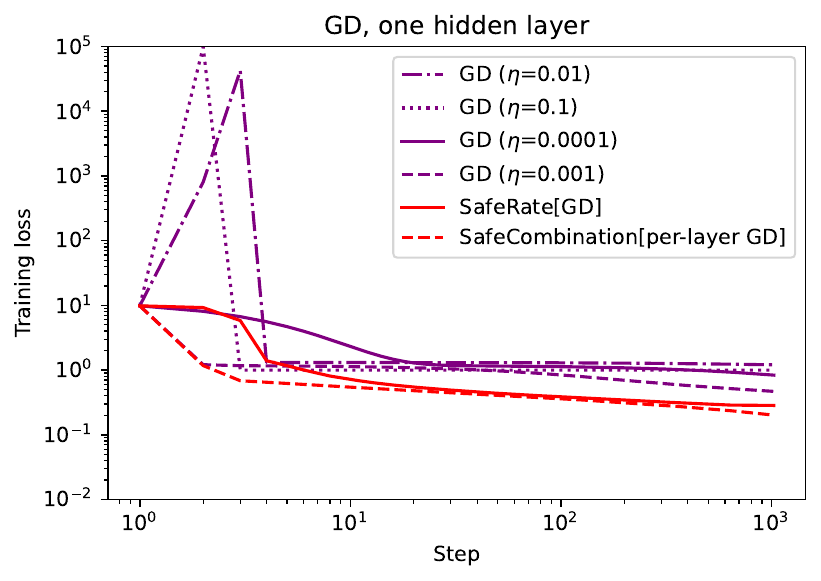}
\includegraphics[width=0.4\linewidth]{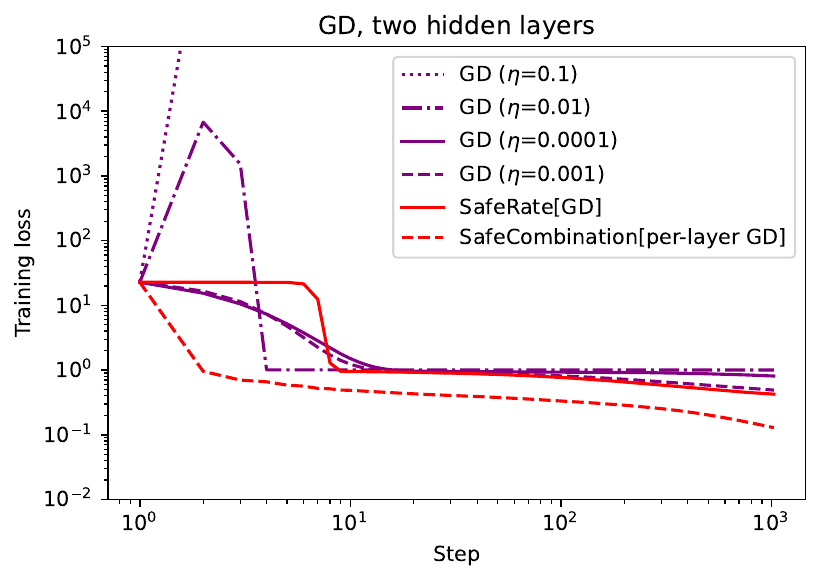}
\includegraphics[width=0.4\linewidth]{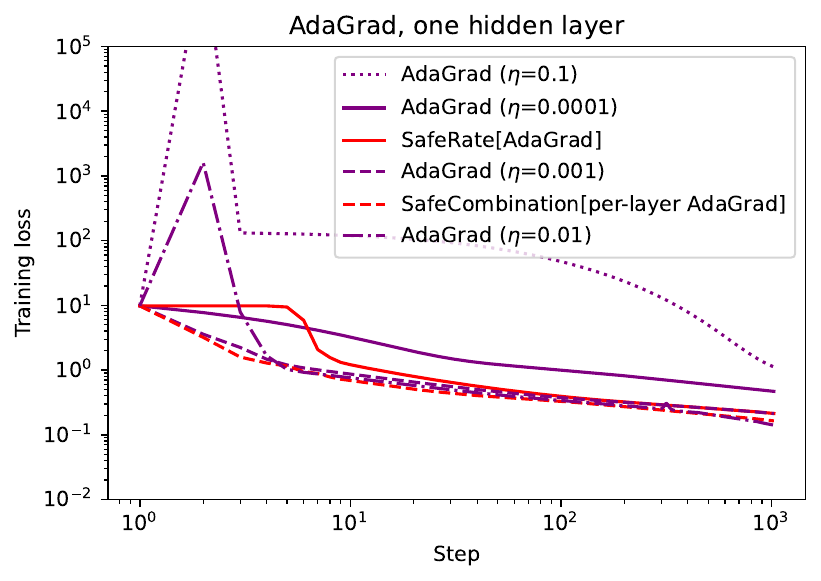}
\includegraphics[width=0.4\linewidth]{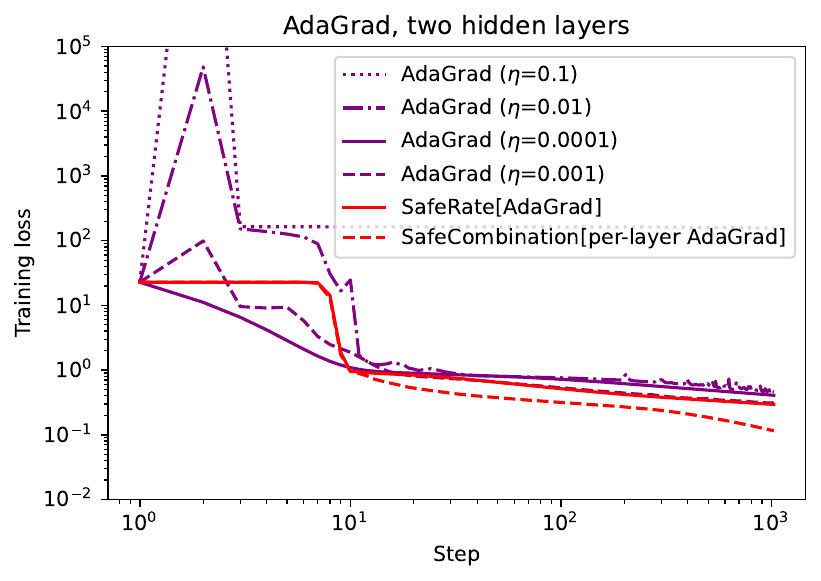}
\includegraphics[width=0.4\linewidth]{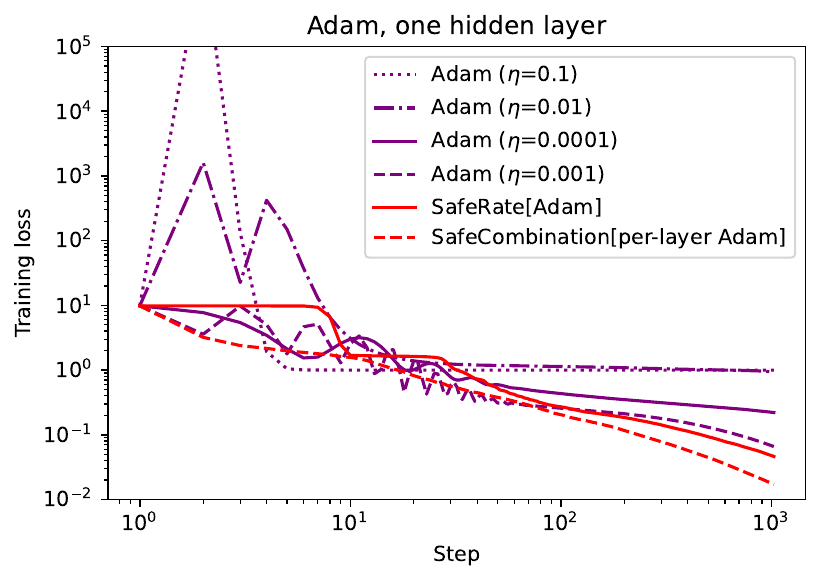}
\includegraphics[width=0.4\linewidth]{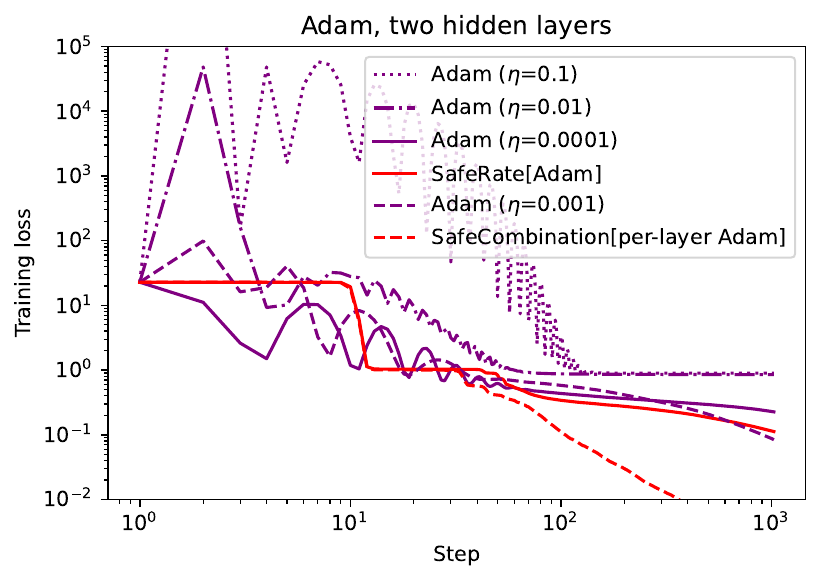}
\caption{Comparison of \saferate and \safecombination to optimizers with a learning rate hyperparameter, when used to optimize a multi-layer perceptron on a subset of the MNIST dataset.
 \safecombination makes more progress per step than \saferate, but also requires significantly more computation per step.
\emph{Using our current implementation, SafeCombination[per-layer Adam] requires 13x more wall time per step than Adam for the one-hidden-layer problem, and 103x more wall time per step for the two-hidden-layer problem (see Figure~\ref{fig:mnist_comparison}).}
}
\label{fig:mnist_saferate_vs_grid}
\end{center}
\end{figure}

Remarkably, we will see in \S\ref{sec:tightness} that for these neural networks, the quadratic bounds used by SafeRate[GD] are nearly tight.  Thus, on these problem, SafeRate[GD] greedily chooses a learning rate that is very close to the one that maximally reduces the loss on each step (similar to its behavior in the linear regression experiment).

\subsubsection{Training an Overparameterized Neural Network in One Step}

In \S\ref{sec:tightness}, we will see that the quadratic bounds used by SafeRate[GD] are \emph{empirically tight} for single-hidden-layer networks with squared error loss: the upper and lower bounds are nearly identical, with the actual loss sandwiched in between them.
Though we have not yet analyzed this phenomenon formally, we would expect the bounds to be tight when the trust region is small enough that the activations are near-affine as a function of the learning rate, for all learning rates in the trust region.  At the same time, theory suggests that gradient descent can train sufficiently wide neural networks without ever departing from a small region of parameter space where the activations are near-affine \cite{lee2019wide}.  Taken together, these observations suggest that \saferate may become very efficient as the width of the network increases.

As a partial confirmation of this conjecture, we now show that there exist (contrived) neural network optimization problems that require hundreds of steps to solve using the Adam optimizer, but that \safecombination is able to solve after just \emph{one step}.  To show this, we consider the same setup as in the previous experiment, but with the size of the training set reduced from 1000 examples to just \emph{one example}, and the number of hidden units increased from 1000 to $10^5$ (and a single hidden layer).  Figure~\ref{fig:highly_overparameterized} compares the performance of \saferate and \safecombination to various baseline optimizers, again considering the best hyperparameter value for each baseline optimizer.\footnote{For this experiment, the baseline optimizers benefit from lower learning rates, so we used a larger grid than the one given in Table~\ref{tab:optimizers}.}

\begin{figure}[H]
\begin{center}
\includegraphics[width=0.9\linewidth]{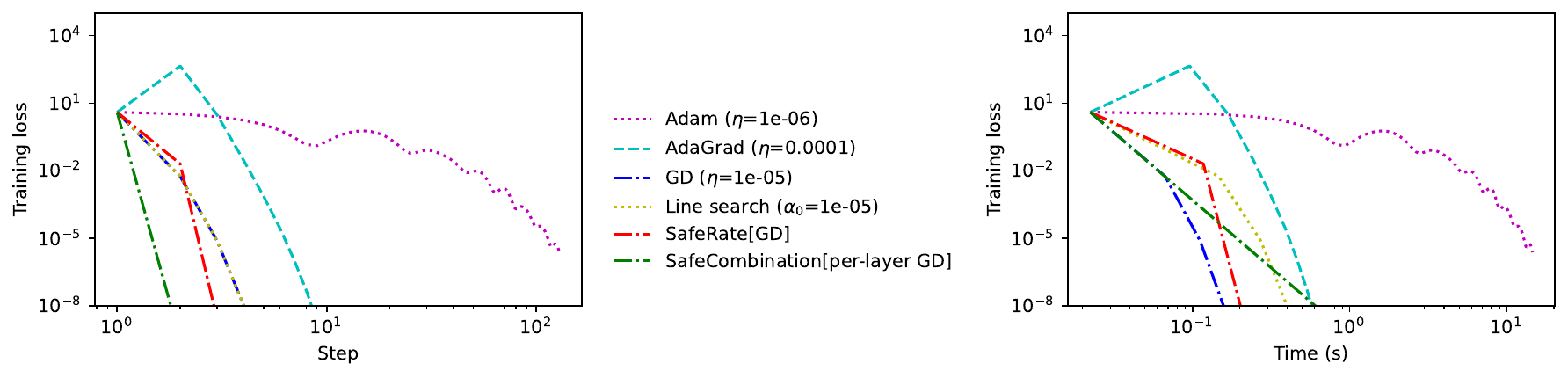}
\caption{Comparison of optimizers on a highly overparameterized problem with only a single training example.  \safecombination is able to solve the problem in just one step, whereas Adam requires hundreds of steps, even using the best learning rate from a coarse grid.}
\label{fig:highly_overparameterized}
\end{center}
\end{figure}

We note that, with a single training example, the optimization problem becomes trivial, and one could imagine other methods that would reach a global minimum in one step.  Nevertheless, the fact that our universal MM optimizers recover good performance in this extreme case is encouraging, and suggests theoretical analysis of their behavior in the infinite-width limit as an interesting area of future work.

\subsection{Tightness of Automatically-Derived Majorizers} \label{sec:tightness}

The effectiveness of any MM optimization algorithm depends on the tightness of the majorizers (i.e., how close the upper bound is to the actual loss).
For least-squares linear regression, it can be shown that the upper bounds used by \saferate are exact.  For more complex losses, the tightness of the upper bounds must be measured empirically.

Figure~\ref{fig:mnist_majorizers} plot the majorizers used by the first iteration of \saferate for the MNIST classification problem discussed in \S\ref{sec:mlps}, for multi-layer perceptrons with 1, 4, 16, or 64 hidden layers.  Each plot shows quadratic, cubic, and quartic majorizers.

With a single hidden layer (top plot), even the quadratic upper bounds are empirically tight.  This is perhaps surprising considering that the loss itself is not quadratic, due to the $\mathrm{softplus}$ hidden layer.  However, the function $h_t(\eta) = f(\vx_t - \eta \nabla f(\vx_t))$ is empirically very close to quadratic for relevant values of $\eta$ (those that are small enough to reduce the loss).

As the number of hidden layers increases, the quadratic bounds become looser, but this can be compensated for by increasing the polynomial degree.  Even with 64 hidden layers, the quartic polynomial majorizer is nearly tight.  Thus, even for very deep networks, \saferate is able to compute a learning rate that near-maximally reduces the loss, at the cost of only a single additional forward pass.

\begin{figure}
\begin{center}
\includegraphics[width=0.45\linewidth]{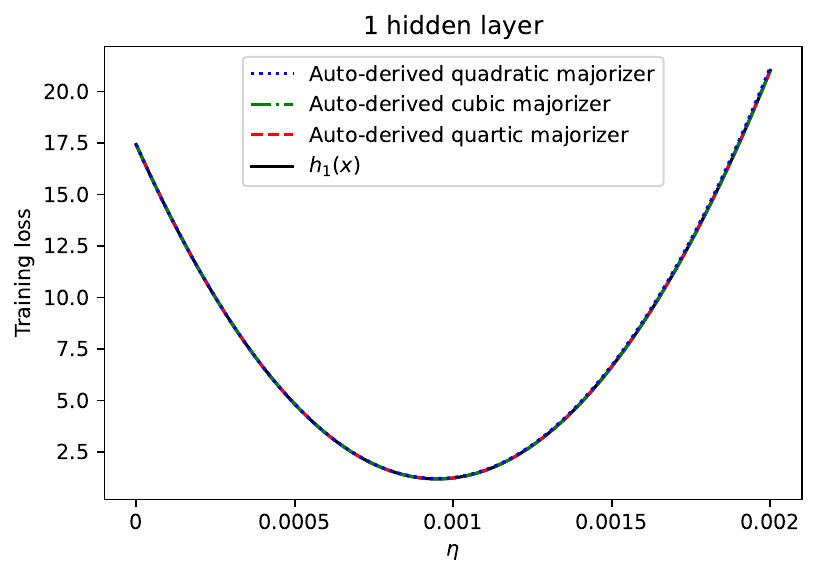}
\includegraphics[width=0.45\linewidth]{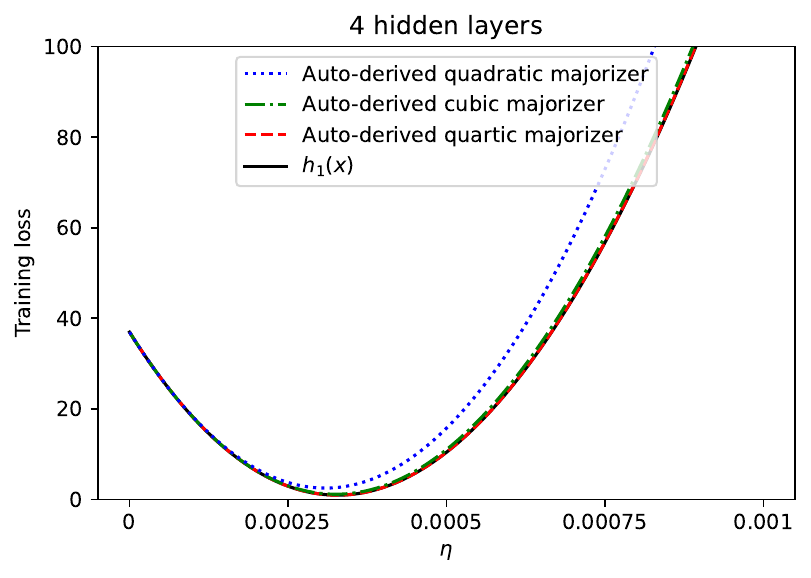}
\includegraphics[width=0.45\linewidth]{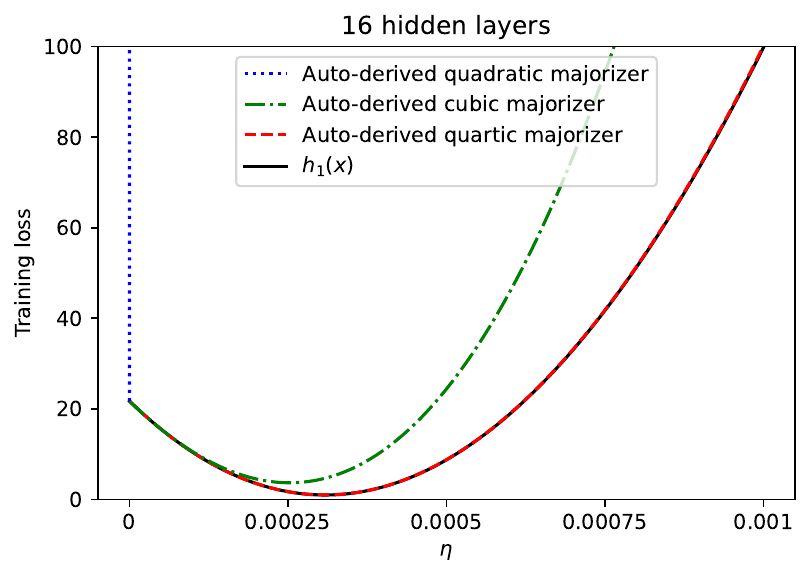}
\includegraphics[width=0.45\linewidth]{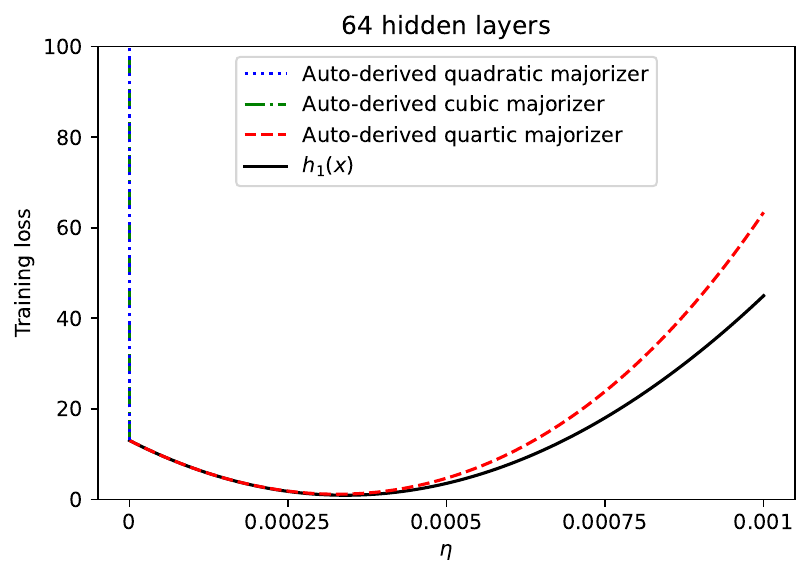}
\end{center}
\caption{Automatically-derived majorizers for the function $h_1(\eta) = f(\vx_1 - \eta \nabla f(\vx_1))$, where $f$ is the training loss for a multi-layer perceptron on a subset of
 the MNIST dataset.}
\label{fig:mnist_majorizers}
\end{figure}

Figure~\ref{fig:autobound_vs_ia_mnist} compares the quadratic majorizers derived by \autobound to quadratic majorizers derived from a bound on the range of the second derivative.  Recall from \S\ref{sec:bounding_remainder} that by bounding the range of the second derivative of a function over a given trust region, we can obtain quadratic bounds on the function which hold over the trust region, and recall that the range of the derivative can be bounded by evaluating the derivative using interval arithmetic.  We consider a variant of this baseline method that handles bilinear operations efficiently using the interval bound propagation scheme described in \cite{streeter2023automatically} (which \autobound also uses).

With one hidden layer, both methods produce nearly tight majorizers.  However, as the number of hidden layers increases, the majorizers produced by the baseline method quickly become much looser than those produced by \autobound.  With three hidden layers, the majorizer produced by the baseline method yields a learning rate that is over 11 times smaller than the one obtained using the majorizer produced by \autobound.  With four hidden layers, the learning rate produced by the baseline method is over 2500 times smaller than the one obtained using \autobound.

\begin{figure}
\begin{center}
\includegraphics[width=0.45\linewidth]{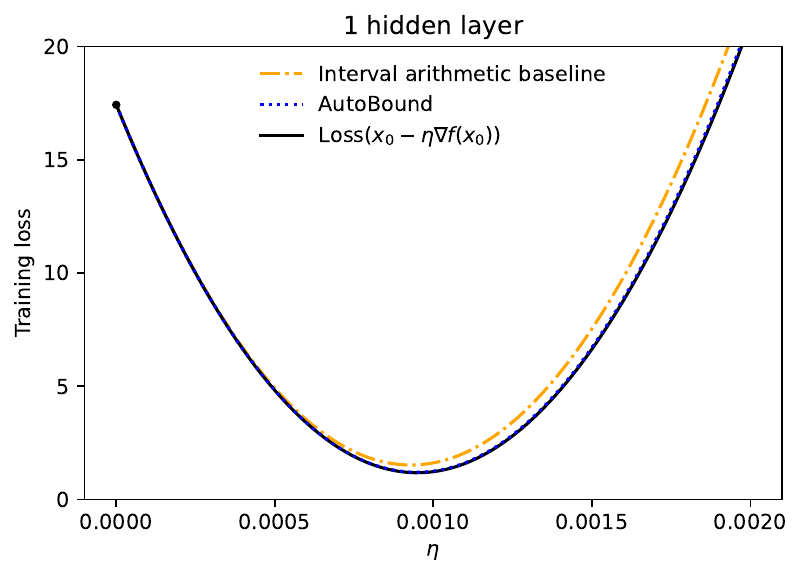}
\includegraphics[width=0.45\linewidth]{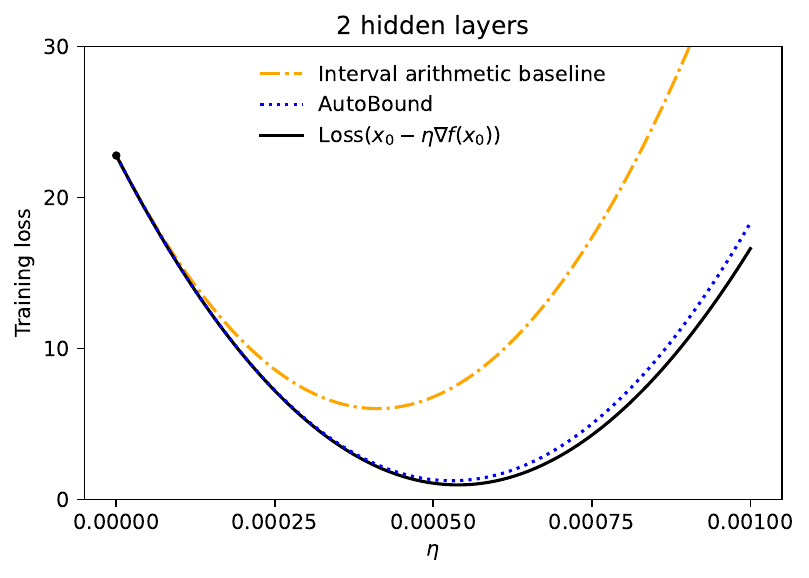}
\includegraphics[width=0.45\linewidth]{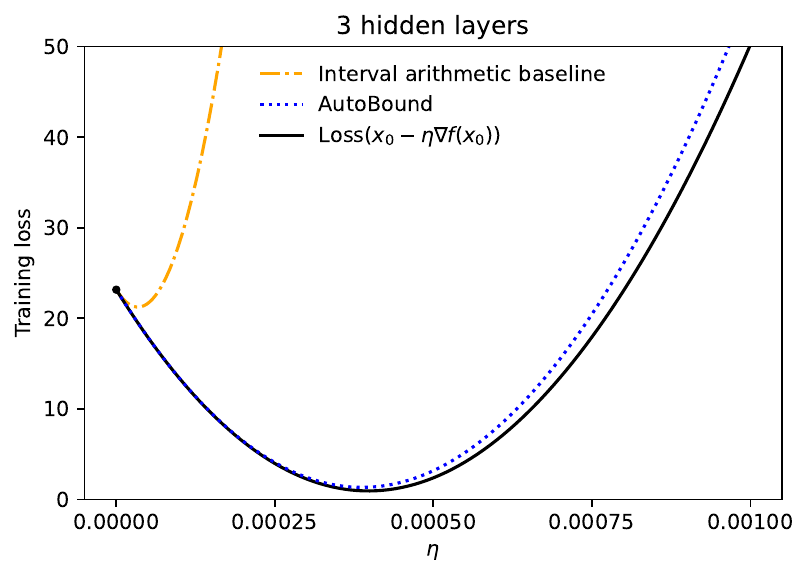}
\includegraphics[width=0.45\linewidth]{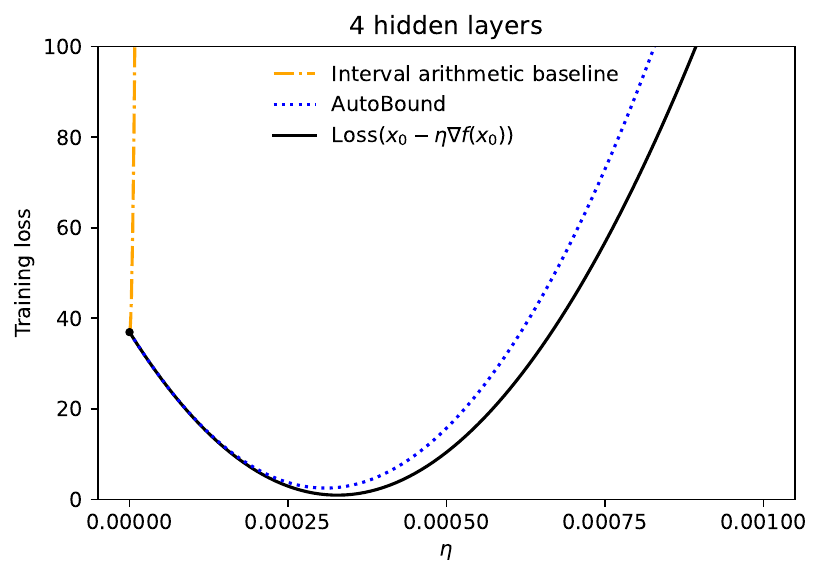}
\end{center}
\caption{Comparison of quadratic majorizers derived by \autobound to those derived using a baseline method based on range-bounding the second derivative.  Each plot shows majorizers for the function $h_1(\eta) = f(\vx_1 - \eta \nabla f(\vx_1))$, where $f$ is the training loss for a multi-layer perceptron on a subset of
 the MNIST dataset.}
\label{fig:autobound_vs_ia_mnist}
\end{figure}

\subsection{Summary}

In this section, we evaluated \saferate and \safecombination on a variety of full-batch optimization problems.  Some of our most significant observations are:
\begin{itemize}
  \item Applied to random linear and logistic regression problems, SafeRate is able to boost the performance of gradient descent and Adam while simultaneously eliminating the need to tune a learning rate hyperparameter.  For linear regression with non-normal errors, SafeRate can \emph{dramatically} outperform tuned versions of the baseline optimizers.
  \item There exist optimization problems where \saferate converges at a faster rate than gradient descent, Adam, or AdaGrad (in particular, applied to a  one-dimensional quartic, \saferate converges at a super-linear rate, whereas gradient descent, Adam, and AdaGrad appear to converge linearly).  It would be very interesting to characterize the space of such problems theoretically.
  \item For multi-layer perceptrons with squared error loss, the bounds computed by \saferate are \emph{empirically tight} (for up to 64 hidden layers, when using quartic bounds), and thus \saferate uses the learning rate that maximally reduces the loss at each step.  As a result, SafeRate can boost the performance of gradient descent while eliminating the learning rate hyperparameter, as was the case with linear regression.
\end{itemize}

We also saw that \safecombination can dramatically outperform \saferate on a per-step basis, at the cost of significant additional computation per step.  In the extreme case of a wide single-hidden-layer network and a single training example, \safecombination was able to reach a global minimum in \emph{one step}.  Our evaluation of \safecombination considered only a few possible choices for the matrix of update directions, and further exploration of this space remains a promising area of future work.

\section{Related Work} \label{sec:mm_related}

At a high level, this chapter considers the problem of minimizing a scalar-valued function $f: \reals^n \to \reals$, where the function $f$ is available in symbolic form, and is composed of sub-differentiable elementary functions.  This problem has been studied in the applied mathematics community for at least 80 years and is the subject of a vast literature; see the textbook by Nocedal and Wright \cite{nocedal1999numerical} for an introduction.

Most relevant to our work is the literature on majorization-minimization (MM) optimizers, which iteratively reduce the loss by minimizing a locally-tight upper bound (called a majorizer).  MM is itself the subject of a large literature; see \cite{hunter2004tutorial} for a tutorial, and see the textbooks by Lange \cite{lange2016mm} and de Leeuw \cite{de2016block} for a through introduction.  Majorizers have been derived by hand for many specific problems of interest, including 
logistic regression \cite{bohning1988monotonicity}, quantile regression \cite{hunter2000quantile}, multidimensional scaling \cite{de1988convergence,groenen1995majorization,de2009multidimensional}  generalized Bradley-Terry models \cite{hunter2004mm}, and support vector machines \cite{groenen2008svm}.

However, there seems to be very little work on deriving the majorizer automatically via symbolic computation, as we have done.  The first work we are aware of is \cite{tran2015fast}, which derives a quadratic majorizer using a recursive procedure that can be applied to neural networks of any depth.  However, the approach taken by \cite{tran2015fast} differs from ours in several critical ways:
\begin{enumerate}
  \item The algorithm of \cite{tran2015fast} seeks a quadratic majorizer that is valid everywhere (not just over a specified trust region), and hence is not applicable to losses such as $f(x) = (x - 3)^4$, which grow faster than any quadratic.  As a consequence, the algorithm of \cite{tran2015fast} can only derive a quadratic majorizer for a neural network loss when the weights in all but one layer are held constant.
  \item The algorithm of \cite{tran2015fast} propagates bounds from the output of the network toward the input, analogous to reverse-mode automatic differentiation.  Such an algorithm has the virtue of efficiently computing majorizers that depend on a large number of input variables.  In contrast, \autobound uses memory linear in the number of input variables, making it only practical for majorizers that are a function of some lower-dimensional quantity such as the learning rate.  However, the efficiency of the reverse-mode algorithm comes at a price, as the backward propagation of quadratic bounds requires the use of inequalities that can become very loose as the network grows wide.  %
  \item As presented, the algorithm of \cite{tran2015fast} is specific to training a neural network with squared error loss and hyperbolic tangent activation functions, although we believe it could be generalized to other activation functions.  %
\end{enumerate}  

As discussed in \cite{streeter2023automatically}, we plan to develop a reverse-mode variant of \autobound in the future, which could be used to derive majorizers that are more directly comparable to the ones derived by \cite{tran2015fast}.

More recently, \cite{bernstein2023automatic} presented a hyperparameter-free gradient descent algorithm for training deep neural networks, based on majorization-minimization ideas.  This work was appeared after the initial publication of our work\footnote{Our work first appeared as Chapter 5 of version 1 of \cite{streeter2023automatically}, which was posted to arXiv in December 2022, while \cite{bernstein2023automatic} was posted to arXiv in April 2023.}, and was done independently.  Unlike our work, this work is based on a number of approximations that do not yield true majorizers.  It also applies to a narrower class of losses, and is not \emph{universal} in the sense we have described.  However, it demonstrates that a hyperparameter-free optimizer based on majorization-minimization ideas can successfully train deep neural networks at ImageNet scale.

\section{Future Work} \label{sec:mm_future}

In this paper, we have presented universal MM optimization algorithms built on top of \autobound \cite{streeter2023automatically}.  Though we believe our experiments have shown that these algorithms exhibit qualitatively new (and desirable) behavior, we have not done all we can to turn them into practical general-purpose optimizers.

Promising areas of future work include:
\begin{itemize}
  \item \emph{Mini-batch optimization.}  Both the theory and experiments in this chapter have been limited to full-batch optimization.  However, the MM paradigm can be extended to mini-batch optimization (e.g., \cite{mairal2013stochastic}).  A universal mini-batch MM optimization algorithm may outperform Adam and AdaGrad on certain large-scale machine learning problems.
  \item \emph{Approximate majorizers.}  Currently, both \saferate and \safecombination require more wall time per step that Adam or AdaGrad.  However, the wall time can be made comparable to Adam or AdaGrad by computing the majorizer approximately, using a random subset of the training data.  Preliminary experiments show that this can yield better tradeoffs between loss and wall time early in optimization, but some care is required to ensure convergence to the same loss asymptotically.
\end{itemize}

Additionally, the universal MM optimization algorithms presented in this chapter have been limited to majorizing the loss as a function of a small number of learning rates.  This was necessary because \autobound, like forward-mode automatic differentiation, requires memory linear in the number of inputs.  However, as discussed in \cite{streeter2023automatically}, we plan to develop a reverse-mode variant of \autobound that would allow us to efficiently compute quadratic majorizers for a function $f: \reals^n \to \reals$ as a function of the $n$-dimensional input.  This would allow us to define universal MM optimization algorithms that determine the update direction on their own, rather than taking it as a given (as \saferate does).

\bibliographystyle{plainnat}
\bibliography{umm}

\begin{appendices}

\section{Proofs} \label {sec:proofs}

We now provide proofs for the theorems stated in \S\ref{sec:theory}.

We first prove Theorem~\ref{thm:global_convergence}, which shows that as $T \rightarrow \infty$, \saferate and \safecombination are guaranteed to find a point with vanishingly small gradient norm.

\begin{customthm}{1}
\thmglobalconvergence
\end{customthm}
\begin{proof}
\thmglobalconvergenceproof
\end{proof}

We now prove Theorem~\ref{thm:time_complexity}
\begin{customthm}{2}
\thmtimecomplexity
\end{customthm}
\begin{proof}
\thmtimecomplexityproof
\end{proof}

\section{Implementation Details} \label{sec:implementation_details}

We now provide details about the \saferate and \safecombination algorithms that were omitted from the main text.  We first describe how to derive majorizers for multivariate functions, which is necessary for the direct approach mentioned in \S\ref{sec:direct_approach}, and also for \safecombination.

\subsection{Automatically Deriving Multivariate Majorizers}

To describe how the AutoBound algorithm can be used to compute majorizers for multivariate functions, we first need to introduce some additional notation.

\subsubsection{Preliminaries}

If $\mA$ is a tensor of rank $r$, and $\mB$ is a tensor of rank $q \le r$, the inner product $\inner{\mA}{\mB}$ is a tensor of rank $r - q$, whose elements are defined by
\begin{align}
  & \inner{\mA}{\mB}_{i_1, i_2, \ldots, i_{r-q}} \nonumber \\
  & \eqdef \sum_{(j_1, j_2, \ldots, j_q) \in \mathrm{indices}(\mB)} \mA_{i_1, i_2, \ldots, i_{r-q}, j_1, j_2, \ldots, j_q}  \mB_{j_1, j_2, \ldots, j_q}. \label{eq:inner}
\end{align}
Observe that if $\a$ and $\b$ are vectors, $\inner{\a}{\b}$ is the usual dot product, while if $\mA$ is a matrix and $\b$ is a vector, then $\inner{\mA}{\b}$ is the usual matrix-vector product.

If $\mA$ is a tensor of rank $r$, and $\mB$ is a tensor of rank $s$, the outer product $\mA \otimes \mB$ is a tensor of rank $r + s$, whose elements are defined by
\begin{equation} \label{eq:outer}
  (\mA \otimes \mB)_{i_1, i_2, \ldots, i_r, j_1, j_2, \ldots j_s} = \mA_{i_1, i_2, \ldots, i_r} \mB_{j_1, j_2, \ldots, j_s}.
\end{equation}
For an integer $k \ge 0$, we use $\mA^{\otimes k}$ to denote a repeated outer product: for $k > 0$, $\mA^{\otimes k} \eqdef \mA \otimes \mA^{\otimes k-1}$, while $\mA^{\otimes 0} \eqdef 1$.  Observe that if $\mA$ has rank $r$, then $\mA^{\otimes k}$ has rank $k r$.

\begin{definition} [Tensor interval]
For tensors $\mA, \mB \in \reals^{n_1 \times n_2 \times \ldots \times n_k}$, the \emph{tensor interval} $[\mA, \mB]$ is the set of tensors $\set{\mX \in \reals^{n_1 \times n_2 \times \ldots \times n_k}: \mA \le \mX \le \mB}$, where the inequality is elementwise.
\end{definition}

Let us adopt the convention that indexing a tensor interval gives a scalar  interval.  For example if $\sA = [\lep{\sA}, \rep{\sA}]$ has rank $r$, then $\sA_{i_1, i_2, \ldots, i_r}$ denotes the scalar interval $[\lep{\sA}_{i_1, i_2, \ldots, i_r}, \rep{\sA}_{i_1, i_2, \ldots, i_r}]$.
With this indexing convention in place, the inner product definition \eqref{eq:inner} immediately generalizes to the case where one or both arguments are tensor intervals (rather than tensors).

Finally, for a vector-variate function $f: \reals^n \to \reals$, the $i$th derivative, written $\nabla^i f(\vx)$, is a tensor of rank $i$ whose elements are
\begin{equation}
  \nabla^i f(\vx)_{j_1, j_2, \ldots, j_i} \eqdef \frac { \partial^i f(\vx) } { \partial \vx_{j_1} \partial \vx_{j_2} \ldots \partial \vx_{j_i} }.
\end{equation}

\subsubsection{Multivariate Majorizers}

When applied to a vector-variate function $f: \reals^d \to \reals$, AutoBound takes as input a vector interval $[\va, \vb] \subseteq \reals^d$, a reference point $\vx_0 \in \reals^d$, and a polynomial degree $k$, and returns a tensor interval $\sI = [\lep{\sI}, \rep{\sI}]$ such that for all $\vx \in [\va, \vb]$,
\begin{equation}
  f(\vx) \in \paren { \sum_{i=0}^{k-1}  \inner{ \nabla^i f(\vx_0) } {(\vx - \vx_0)^{\otimes i} } } + \inner { \sI } { (\vx - \vx_0)^{\otimes k} }.
\end{equation}

If $\vx - \vx_0 \ge 0$ (elementwise) for all $\vx \in [\va, \vb]$, then this implies the upper bound:
\begin{equation}
  f(\vx) \le \paren { \sum_{i=0}^{k-1}  \inner{ \nabla^i f(\vx_0) } {(\vx - \vx_0)^{\otimes i} } } + \inner { \rep{\sI} } { (\vx - \vx_0)^{\otimes k} }.
\end{equation}
where $\rep{I}$ is the right endpoint of $\sI$.

\subsection{Computing the \saferate Learning Rate}

As described in \S\ref{sec:saferate}, the \saferate algorithm computes a learning rate $\eta_t$ using the formula
\begin{equation}
  \eta_t = \argmin_{\eta \in [0, \etamax_t]} \set { P_t(\eta) }
\end{equation}
where $P_t$ is a degree $k$ polynomial.

In the special case $k = 2$, letting $\rep{I_t}$ denote the quadratic coefficient of $P_t$, we have
\begin{equation} \label{eq:optimal_eta}
  \eta_t = \begin{cases}
  \max \set {0, \min \set { - \frac { \nabla f(\vx_t)^\tee \vv_t} {2 \rep{I_t}}, \etamax_t } } & \rep{I_t} > 0 \\
  \argmin_{\eta \in \set{0, \etamax_t}} \set { \eta \nabla f(\vx_t)^\tee \vv_t + \rep{I_t} \eta^2} & \rep{I_t} \le 0.
\end{cases}
\end{equation}
The solution in \eqref{eq:optimal_eta} for the case $\rep{I_t} > 0$ is obtained by minimizing the quadratic (by setting the derivative to 0) and projecting the minimizer onto the interval $[0, \etamax_t]$.  The solution for $\rep{I_t} \le 0$ holds because the minimizer of a concave quadratic over an interval is always one of the two endpoints.

More generally, the minimum value of a polynomial over a closed interval must either be one of the end points, or a point at which the derivative is zero.  To find the points at which the derivative is zero, we must find the roots of a degree $k - 1$ polynomial, which can be done by computing the eigenvalues of the companion matrix \cite{press2007numerical}.
We then obtain the argmin by minimizing $P_t$ over this finite set of points (the roots of the derivative, plus the two endpoints of the interval, namely $0$ and $\etamax_t$).

\subsection{Optimizing a Non-Convex Quadratic Over a Hyperrectangle} \label{sec:minimizing_quadratic}

In \S\ref{sec:safe_combination} we presented the \safecombination algorithm, each step of which requires optimizing a possibly non-convex quadratic function over a hyperrectangle.  Although this is in general NP-hard (as discussed below), we provide an efficient algorithm that maintains good performance in the case where the quadratic function is convex, while still offering a desirable performance guarantee in the general case.

\newcommand{\xmin}{\lep{\vec{x}}}
\newcommand{\xmax}{\rep{\vec{x}}}

For a vector $\vb \in \reals^d$, and a symmetric matrix $\mA \in \reals^{d \times d}$, let $f: \reals^d \to \reals$ be defined by
\begin{equation}
  f(\vx) \eqdef \frac 1 2 \vx^\tee \mA \vx - \vb^\tee \vx.
\end{equation}
For vectors $\xmin, \xmax \in \reals^d$, with $\xmin \le \xmax$ (elementwise), we wish to minimize $f$ over the hyperrectangle $[\xmin, \xmax]$.  That is, we wish to find
\begin{equation} \label {eq:xstar}
  \vx^* \eqdef \argmin_{\vx \in [\xmin, \xmax]} \set { f(\vx) }.
\end{equation}
If $\mA$ is arbitrary, this problem includes the maximum cut problem as a special case, as is therefore NP-hard \cite{de2008complexity}.
If $\mA$ is positive semidefinite, then $f(\vx)$ is convex, and $\vx^*$ can be found in polynomial time using a variety of convex optimization methods.  If $\mA$ is positive definite, then the global minimum of $f(\vx)$ can be found using $d$ steps of conjugate gradient descent.  If the global minimum happens to be inside the hyperrectangle $[\xmin, \xmax]$, then it must equal $\vx^*$.

We now describe a variant of conjugate gradient descent that always produces a solution inside the hyperrectangle $[\xmin, \xmax]$, while maintaining some desirable guarantees of conjugate gradient descent in the case where $\mA$ is positive definite.  To explain our method, recall that two vectors $\vu, \vv \in \reals^d$ are said to be \emph{conjugate} with respect to $\mA$ if
\begin{equation}
  \vu^\tee \mA \vv = 0.
\end{equation}
Like the conjugate gradient method, our algorithm will iteratively compute a set of conjugate directions $\set{\vp_1, \vp_2, \ldots, \vp_d}$, and will output a linear combination of the conjugate directions.  However, this linear combination will be constrained to lie within the hyperrectangle $[\xmin, \xmax]$.

Our algorithm is defined by the following equations:
\begin{align}
  \vx_0 & = \zeros_d \nonumber \\
  \vr_i & =  \nabla f(\vx_i) & \forall i \in \set{0, 1, \ldots, d} \nonumber \\
  \vp_i & = -\vr_i + \frac { \vr_i^\tee \vr_i } { \vr_{i-1}^\tee \vr_{i-1} } \vp_{i-1} & \forall i \in \set{1, 2, \ldots, d} \nonumber \\
  \alpha_i & = \argmin_{\alpha: \vx_{i-1} + \alpha_i \vp_i \in [\xmin, \xmax]} \set { f(\alpha \vp_i) } & \forall i \in \set{1, 2, \ldots, d} \nonumber \\
  \vx_i & = \vx_{i-1} + \alpha_i \vp_i & \forall i \in \set{1, 2, \ldots, d} \label {eq:cg_variant}
\end{align}

The vectors $\vr_i$ and $\vp_i$ are defined exactly as in the conjugate gradient descent algorithm (see Algorithm 5.2 of \cite{nocedal1999numerical}).  The only difference is that $\alpha_i$ is defined so as to guarantee that $\vx_i$ remains in the hyperrectangle.

It easy to see by inspection that the vectors $\vp_i$ and $\vp_{i-1}$ are conjugate with respect to $\mA$.  With additional work, it can be shown that all $\vp_i$ and $\vp_j$ with $i \neq j$ are conjugate with respect to $\mA$.

\begin{lemma} [Theorem 5.3 of \cite{nocedal1999numerical}] \label {lem:conjugate}
For any symmetric matrix $\mA$ (not necessarily positive definite), the set of vectors $\set{\vp_1, \vp_2, \ldots, \vp_d}$ are conjugate with respect to $\mA$ (i.e., $\vp_i^\tee \mA \vp_j$ for $i \neq j$).
\end{lemma}
\ignore{
\begin{proof}
The proof is by induction on $d$.  For $d = 1$, the lemma is trivially true.

Suppose the lemma holds for sets of size $d - 1$.  Then, for any $i < d$,
\begin{align}
  \vp_d^\tee \mA \vp_i
  & = \paren{ \vr_d - \sum_{a=1}^{d-1}  \frac { \vr_i^\tee \mA \vp_a } {\vp_a^\tee \mA \vp_a  } \vp_a }^\tee \mA \vp_i \nonumber \\
  & = \vr_d^\tee \mA \vp_i -  \sum_{a=1}^{d-1}  \frac { \vr_i^\tee \mA \vp_a } {\vp_a^\tee \mA \vp_a  } \vp_a^\tee \mA \vp_i \nonumber \\
  & = \vr_d^\tee \mA \vp_i - \frac { \vr_i^\tee \mA \vp_i } {\vp_i^\tee \mA \vp_i  } \vp_i^\tee \mA \vp_i \nonumber \\
  & = 0
\end{align}
where the third equality uses the fact that $\vp_a$ and $\vp_i$ are conjugate under the induction hypothesis.
\end{proof}
}

With the conjugacy of $\set{\vp_1, \vp_2, \ldots, \vp_d}$ established, we can prove the following theorem, which shows that the algorithm always monotonically reduces the loss, while maintaining the good performance of conjugate gradient descent if $\mA$ is positive definite and the hyperrectangle is sufficiently large.

\begin{theorem}
For any symmetric matrix $\mA$, the sequence of iterates returned by algorithm \eqref{eq:cg_variant} satisfy
\[
  f(\vx_d) \le f(\vx_{d-1}) \le \ldots \le f(\vx_1) \le 0.
\]
Furthermore, if $\mA$ is positive definite, and each $\vx_i$ is in the interior of $[\xmin, \xmax]$, then $\vx_d = \vx^*$.
\end{theorem}
\begin{proof}
First observe that, because $\vp_a$ and $\vp_b$ are conjugate for $a \neq b$ (by Lemma~\ref{lem:conjugate}), the objective function decomposes additively:
\begin{align}
  f(\vx_i)
  & = f\paren{\sum_{a=1}^i \alpha_a \vp_a} \nonumber \\
  & = \frac 1 2 \paren { \sum_{a=1}^i \alpha_a \vp_a }^\tee \mA \paren { \sum_{b=1}^i \alpha_b \vp_b } - \vb^\tee \paren { \sum_{a=1}^i \alpha_a \vp_a } \nonumber \\
  & = \frac 1 2 \paren { \sum_{a=1}^d \alpha_a \vp_a^\tee \mA \vp_a } - \vb^\tee \paren { \sum_{a=1}^i \alpha_a \vp_a } \nonumber \\
  & = \sum_{a=1}^i f(\alpha_a \vp_a). \label{eq:additively_separable}
\end{align}
This immediately implies
\begin{equation} \label{eq:f_one_more_term}
  f(\vx_i) = f(\vx_{i-1}) + f(\alpha_i \vp_i).
\end{equation}
Furthermore, because $\alpha_i = \argmin_{\alpha \in I_i} \set { f(\alpha \vp_i) }$, where $I_i$ is an appropriately-defined interval that includes 0, we have
\begin{equation} \label {eq:f_decrease}
  f(\alpha_i \vp_i) \le f(0 \vp_i) = 0.
\end{equation}
Combining \eqref{eq:f_one_more_term} and \eqref{eq:f_decrease} gives $f(\vx_i) \le f(\vx_{i-1})$, which completes the first part of the proof.

To prove the second part of the theorem, we first note that, using \eqref{eq:additively_separable},
\begin{align}
  \min_{\alpha_1, \alpha_2, \ldots, \alpha_d \in \reals} \set { f\paren{\sum_{i=1}^d \alpha_i \vp_i} }
  & =  \min_{\alpha_1, \alpha_2, \ldots, \alpha_d \in \reals} \set { \sum_{i=1}^d f(\alpha_i \vp_i) } \nonumber \\
  & = \sum_{i=1}^d \min_{\alpha_i \in \reals} \set { f(\alpha_i \vp_i) }. \label{eq:minimize_each_term}
\end{align}
Now suppose that each $\vx_i$ is in the interior of $[\xmin, \xmax]$.  It follows that for all $i$, $\alpha_i$ is in the interior of $I_i$.  Because $f(\alpha \vp_i)$ is a quadratic function of $\alpha$, the fact that $\alpha_i$ is in the interior of $I_i$ implies $\alpha_i = \argmin_{\alpha \in \reals} \set { f(\alpha_i \vp_i) }$.  Using \eqref{eq:minimize_each_term}, this implies
\begin{equation} \label {eq:xd_argmin}
  \vx_d = \argmin_{\alpha_1, \alpha_2, \ldots, \alpha_d \in \reals} \set { f\paren{\sum_{i=1}^d \alpha_i \vp_i} }.
\end{equation}
To complete the proof, we note that if $\mA$ is positive definite and the set of vectors $\set{\vp_1, \vp_2, \ldots, \vp_d}$ are conjugate with respect to $\mA$, then it is not hard to see that the vectors $\vp_1, \vp_2, \ldots, \vp_d$ must be linearly independent, and therefore span $\reals^d$.  It follows that the right hand side of \eqref{eq:xd_argmin} is the global minimizer, $\vx^*$, and therefore $\vx_d = \vx^*$, as claimed.
\end{proof}

In experimenting with this algorithm, we found that it sometimes returns $\zeros$ in cases where a simpler coordinate descent algorithm is able to make progress.  This happens when reducing $f$ by moving along any conjugate direction $\vp_i$ results in a violation of the constraint $\vx \in [\lep{\vx}, \rep{\vx}]$, but there exists a Euclidean basis vector $\ve_i$ such that, for some $\alpha_i \in \reals$, $f(\alpha_i \ve_i) < 0$ and $\alpha_i \ve_i \in [\lep{\vx}, \rep{\vx}]$.  To improve performance in such cases, we follow \eqref{eq:cg_variant} by a single pass of cyclic coordinate descent in our implementation of \safecombination.

\end{appendices}

\end{document}